\documentclass[12pt, reqno]{amsart}
\usepackage{amsmath, amsthm, amscd, amsfonts, amssymb, graphicx, color}
\usepackage[bookmarksnumbered, colorlinks, plainpages]{hyperref}
\hypersetup{colorlinks=true,linkcolor=red, anchorcolor=green, citecolor=cyan, urlcolor=red, filecolor=magenta, pdftoolbar=true}

\newtheorem{theorem}{Theorem}[section]
\newtheorem{lemma}[theorem]{Lemma}
\newtheorem{proposition}[theorem]{Proposition}
\newtheorem{corollary}[theorem]{Corollary}

\newtheorem*{TheoremA}{Theorem A}
\newtheorem*{TheoremB}{Theorem B}

\theoremstyle{definition}
\newtheorem{definition}[theorem]{Definition}

\theoremstyle{remark}

\numberwithin{equation}{section}

\newcommand\C{\mathbb C}
\newcommand\D{\mathbb D}

\newcommand\rtkmu{R^t(K, \mu)}

\renewcommand\i{\infty}
\newcommand\area{{\frak m}}
\newcommand\CT{{\mathcal C}}

\newcommand{\quotes}[1]{``#1''}

\begin{document}

\setcounter{page}{1}

\title[Mean Rational Approximation]{Mean Rational Approximation for Compact Subsets with Thin Boundaries}

\author[J. Conway \MakeLowercase{and} L. Yang]{John B. Conway$^1$ \MakeLowercase{and} Liming Yang$^2$}

\address{$^{1}$Department of Mathematics, The George Washington University, Washington, DC 20052}
\email{\textcolor[rgb]{0.00,0.00,0.84}{conway@gwu.edu}}

\address{$^2$Department of Mathematics, Virginia Polytechnic and State University, Blacksburg, VA 24061.}
\email{\textcolor[rgb]{0.00,0.00,0.84}{yliming@vt.edu}}




\begin{abstract} In 1991, J. Thomson \cite{thomson} obtained a celebrated decomposition theorem for $P^t(\mu),$ the closed subspace of $L^t(\mu)$ spanned by the analytic polynomials, when $1 \le t < \i.$ Later, J. Brennan \cite{b08} generalized Thomson's theorem to $R^t(K, \mu),$ the closed subspace of $L^t(\mu)$ spanned by the rational functions with poles off a compact subset $K$ containing the support of $\mu,$ when the diameters of the components of $\mathbb C\setminus K$ are bounded below. We extend the above decomposition theorems for $R^t(K, \mu)$ when the boundary of $K$ is not too \quotes{wild}. 
\end{abstract} \maketitle

\section{\textbf{Introduction}}

For a Borel subset $E$ of the complex plane $\C,$ let $M_0(E)$ denote the set of finite complex-valued Borel measures that are compactly supported in $E$ and let $M_0^+(E)$ be the set of positive measures in $M_0(E).$ The support of $\nu\in M_0(\C),$ $\text{spt}(\nu),$ is the smallest closed set that has full $|\nu|$ measure. 
For a compact subset $K\subset \C,$ $\mu\in M_0^+(K),$ and $1\leq t < \infty$, the analytic polynomials and functions in 
$\mbox{Rat}(K) := \{q:\mbox{$q$ is a rational function 
with poles off $K$}\}$
are members of $L^t(\mu)$. Let $P^t(\mu)$ denote the closure of the (analytic) 
polynomials in $L^t(\mu)$ and let $R^t(K, \mu)$ denote the closure of $\mbox{Rat}(K)$ in $L^t(\mu)$. Denote 
 \[
 \ R^{t,\i}(K, \mu) = R^t(K, \mu) \cap L^\infty(\mu).
 \] 
   
For $z_0\in \C$ and $r> 0,$ let $\D(z_0, r) := \{z:~|z - z_0| < r\}.$
A point $z_0$ in $\mathbb{C}$ is called a \textit{bounded point evaluation} for $P^t(\mu)$ (resp., $R^t(K, \mu)$)
if $f\mapsto f(z_0)$ defines a bounded linear functional, denoted $\phi_{z_0},$ for the analytic polynomials (resp., functions in $\mbox{Rat}(K)$)
with respect to the $L^t(\mu)$ norm. The collection of all such points is denoted $\mbox{bpe}(P^t(\mu))$ 
(resp., $\mbox{bpe}(R^t(K, \mu)$)).  If $z_0$ is in the interior of $\mbox{bpe}(P^t(\mu))$ (resp., $\mbox{bpe}(R^t(K, \mu)$)) 
and there exist positive constants $r$ and $M$ such that $|f(z)| \leq M\|f\|_{L^t(\mu)}$, whenever $z\in \D(z_0, r)$ 
and $f$ is an analytic polynomial (resp., $f\in \mbox{Rat}(K)$), then we say that $z_0$ is an 
\textit{analytic bounded point evaluation} for $P^t(\mu)$ (resp., $R^t(K, \mu)$). The collection of all such 
points is denoted $\mbox{abpe}(P^t(\mu))$ (resp., $\mbox{abpe}(R^t(K, \mu))$). Actually, it follows from Thomson's Theorem
\cite{thomson} (or see Theorem \ref{thmthom}, below) that $\mbox{abpe}(P^t(\mu))$ is the interior of $\mbox{bpe}(P^t(\mu))$. 
This also holds in the context of $R^t(K, \mu)$ as was shown by J. Conway and N. Elias in \cite{ce93}. Now, 
$\mbox{abpe}(P^t(\mu))$ is the largest open subset of $\mathbb{C}$ to which every function $f\in P^t(\mu)$ has an analytic 
continuation 
$\rho(f)(z) := \phi_z (f)$ for $z\in \mbox{abpe}(P^t(\mu)),$
that is, 
$f(z) = \rho(f)(z) ~ \mu|_{\mbox{abpe}(P^t(\mu))}-a.a.,$
and similarly in the context of $R^t(K, \mu)$. The map $\rho$ is called the evaluation map.

Let $S_\mu$ denote the multiplication by $z$ on $P^t(\mu)$ (resp., $R^t(K, \mu)$).
The operator $S_\mu$ is pure if $P^t(\mu)$ (resp., $R^t(K, \mu)$) does not have a direct $L^t$ summand,
  and is irreducible if $P^t(\mu)$ (resp., $R^t(K, \mu)$) contains no non-trivial characteristic functions.
  For an open subset $U \subset \C,$ we let $H^\infty(U)$ denote the bounded analytic functions on $U$.
This is a Banach algebra in the maximum modulus norm, but it is important to endow it also with a weak-star topology.

We shall use $\overline {E}$ to denote the closure of the set $E\subset \C$ and $\text{int}(E)$ to denote the interior of $E.$ Our story begins with celebrated results of J. Thomson, in \cite{thomson}. 

\begin{theorem} (Thomson 1991)
\label{thmthom}
Let $\mu\in M_0^+(\C)$ and suppose that $1\leq t < \infty$.
Then there is a Borel partition $\{\Delta_i\}_{i=0}^\infty$ of $\mbox{spt}(\mu)$ such that 
\[
 \ P^t(\mu ) = L^t(\mu |_{\Delta_0})\oplus \bigoplus _{i = 1}^\infty P^t(\mu |_{\Delta_i})
 \]
and the following statements are true:

(a) If $i \ge 1$, then $S_{\mu |_{\Delta_i}}$ on $P^t(\mu |_{\Delta_i})$ is irreducible.

(b) If $i \ge 1$ and $U_i :=abpe( P^t(\mu |_{\Delta_i}))$, then $U_i$ is a simply connected region and $\Delta_i\subset \overline{U_i}$.

(c) If $S_\mu$ is pure (that is, if $\Delta_0 = \emptyset$) and $U = \text{abpe}(P^t(\mu))$, then the evaluation map $\rho$ is an isometric isomorphism and a weak-star homeomorphism from $P^t(\mu) \cap L^\infty(\mu)$ onto $H^\infty(U)$. 
\end{theorem}

The next result is due to Brennan \cite{b08}, which extends Theorem \ref{thmthom}.

\begin{theorem} (Brennan 2008)
\label{thmBCE}
 Let $\mu\in M_0^+(K)$ for a compact set $K$.
 Suppose that $1\le t < \infty$ and the diameters of the components of $\mathbb C\setminus K$ are bounded below. 
Then there is a Borel partition $\{\Delta_i\}_{i=0}^\infty$ of $\mbox{spt}(\mu)$ and compact subsets $\{K_i\}_{i=1}^\infty$ such that $\Delta_i \subset K_i$ for $i \ge 1$,
 \[
 \ R^t(K, \mu ) = L^t(\mu |_{\Delta_0})\oplus \bigoplus _{i = 1}^\infty R^t(K_i, \mu |_{\Delta_i}),
 \]
and the following statements are true:

(a) If $i \ge 1$, then $S_{\mu |_{\Delta_i}}$ on $R^t(K_i, \mu |_{\Delta_i})$ is irreducible.

(b) If  $i \ge 1$ and $U_i :=\mbox{abpe}( R^t(K_i, \mu |_{\Delta_i}))$, then $U_i$ is connected and $K_i = \overline{U_i}$.

(c) If  $i \ge 1$, then the evaluation map $\rho_i$ is an isometric isomorphism and a weak-star homeomorphism from $R^{t,\i}(K_i, \mu |_{\Delta_i})$ onto $H^\infty(U_i)$.
\end{theorem}

If $B \subset\C$ is a compact subset, then  we
define the analytic capacity of $B$ by
\[
\ \gamma(B) = \sup |f'(\infty)|,
\]
where the supremum is taken over all those functions $f$ that are analytic in $\mathbb C_{\infty} \setminus B$ ($\mathbb{C}_\infty := \mathbb{C} \cup \{\infty \}$) such that
$|f(z)| \le 1$ for all $z \in \mathbb{C}_\infty \setminus B$; and
$f'(\infty) := \lim _{z \rightarrow \infty} z[f(z) - f(\infty)].$
The analytic capacity of a general subset $E$ of $\mathbb{C}$ is given by: 
$\gamma (E) = \sup \{\gamma (B) : B\subset E \text{ compact}\}.
$
We will write $\gamma-a.a.$, for a property that holds everywhere, except possibly on a set of analytic capacity zero.
Good sources for basic information about analytic
capacity are Chapter VIII of \cite{gamelin}, Chapter V of \cite{conway}, and \cite{Tol14}.

We let $\partial_e K$ (the exterior boundary of $K$) denote the union of
   the boundaries of all the components of $\C \setminus K.$ Define
\begin{eqnarray}\label{BOne}
 \ \partial_1 K = \left \{\lambda \in K:~ \underset{\delta\rightarrow 0}{\overline\lim} \dfrac{\gamma(\D(\lambda,\delta)\setminus K)}{\delta} > 0 \right \}.
 \end{eqnarray}
Obviously, $\partial_e K \subset \partial_1 K \subset \partial K.$ If the diameters of  the components of $\C\setminus K$ are bounded below, then there exist $\epsilon_0 > 0$ and $\delta_0 > 0$ such that for each $\lambda \in \partial K,$
\begin{eqnarray}\label{CDiamBBEqn}
\ \gamma (\D(\lambda,\delta)\setminus K) \ge \epsilon_0 \delta \text{ for }\delta < \delta_0.
\end{eqnarray}
Clearly, if $K$ satisfies \eqref{CDiamBBEqn}, then $\partial K = \partial_1 K.$ Conversely,  it is straightforward to construct a compact subset $K$ such that $\partial K = \partial_1 K$ and $K$ does not satisfy \eqref{CDiamBBEqn}.

In \cite[Section 5]{cy22}, Conway and Yang introduced a concept of non-removable boundary for $R^t(K, \mu )$ to identify $R^{t,\i}(K, \mu)$ with a sub-algebra of $H^\i(\text{int}(K))$ when $K$ is a string of beads set (see \cite[Main Theorem]{cy22}). In this paper, we extend the concept of non-removable boundary, denoted $\mathcal F,$ to an arbitrary compact subset $K$ and $\mu\in M_0^+(K)$ (see Definition \ref{FRDef}  and Definition \ref{FRForRtDef}). Intuitively, the non-removable boundary $\mathcal F$ splits into three sets, $\mathcal F_0$ and $\mathcal F_+ \cup \mathcal F_-$ such that
(1) Cauchy transforms of annihilating measures $g\mu$ ($g\perp R^t(K, \mu)$) are zero on $\mathcal F_0$ and (2) Cauchy transforms of annihilating measures $g\mu$ have zero \quotes{one side non-tangential limits} on $\mathcal F_+ \cup \mathcal F_-.$
We show in Proposition \ref{RTFRProp2} that 
\begin{eqnarray}\label{BOneSubsetF}
\ \partial _1 K \subset \mathcal F, ~\gamma-a.a..
\end{eqnarray}

The aim of this paper is to prove the following main theorem.

\begin{TheoremA}\label{thmA}
 Let $K\subset \C$ be a compact subset, $1\le t < \infty,$ and $\mu\in M_0^+(K).$ Let $\text{abpe}(\rtkmu) = \cup_{i=1}^\i U_i,$ where $U_i$ is a connected component.   
If $\partial U_i \cap \partial  K \subset \mathcal F, ~\gamma-a.a.$ for all $i \ge 1,$
then there is a Borel partition $\{\Delta_i\}_{i=0}^\infty$ of $\mbox{spt}(\mu)$ and compact subsets $\{K_i\}_{i=0}^\infty$ such that $\Delta_i \subset K_i$ for $i \ge 0$,
 \[
 \ R^t(K, \mu ) = R^t(K_0,\mu |_{\Delta_0})\oplus \bigoplus _{i = 1}^\infty R^t(K_i, \mu |_{\Delta_i}),
 \]
and the following properties hold:

(a) $K_0$ is the spectrum of $S_{\mu |_{\Delta_0}}$ and $\text{abpe}(R^t(K_0,\mu |_{\Delta_0})) = \emptyset.$

(b) If $i \ge 1$, then $S_{\mu |_{\Delta_i}}$ on $R^t(K_i, \mu |_{\Delta_i})$ is irreducible.

(c) If  $i \ge 1,$ then $U_i =\mbox{abpe}( R^t(K_i, \mu |_{\Delta_i}))$ and $K_i = \overline{U_i}$.

(d) If  $i \ge 1$, then the evaluation map $\rho_i: f\rightarrow f|_{U_i}$ is an isometric isomorphism and a weak-star homeomorphism from $R^{t, \i}(K_i, \mu |_{\Delta_i})$ onto $H^\infty(U_i)$.
\end{TheoremA}

Consequently, by \eqref{BOneSubsetF}, we prove the following theorem in section 6, which extends both Theorem \ref{thmthom} and Theorem \ref{thmBCE}.

\begin{TheoremB}\label{thmB}
 Let $K$ be a compact set such that $\gamma(\partial K \setminus \partial_1 K) = 0$.
 Suppose that $1\le t < \infty$ and $\mu\in M_0^+(K).$ 
Then there is a Borel partition $\{\Delta_i\}_{i=0}^\infty$ of $\mbox{spt}(\mu)$ and compact subsets $\{K_i\}_{i=1}^\infty$ such that $\Delta_i \subset K_i$ for $i \ge 1$,
 \[
 \ R^t(K, \mu ) = L^t(\mu |_{\Delta_0})\oplus \bigoplus _{i = 1}^\infty R^t(K_i, \mu |_{\Delta_i}),
 \]
and the following statements are true:

(a) If $i \ge 1$, then $S_{\mu |_{\Delta_i}}$ on $R^t(K_i, \mu |_{\Delta_i})$ is irreducible.

(b) If  $i \ge 1$ and $U_i :=\mbox{abpe}( R^t(K_i, \mu |_{\Delta_i}))$, then $U_i$ is connected and $K_i = \overline{U_i}$.

(c) If  $i \ge 1$, then the evaluation map $\rho_i: f\rightarrow f|_{U_i}$ is an isometric isomorphism and a weak-star homeomorphism from $R^{t, \i}(K_i, \mu |_{\Delta_i})$ onto $H^\infty(U_i)$.
\end{TheoremB}

Main Theorem in \cite{cy22} provides an interesting example of $\rtkmu$ that does not satisfy the assumptions of Theorem A (also see \cite[Proposition 3.4]{cy22}).
In section 2, we review some results of analytic capacity and Cauchy transform that are needed in our analysis. Theorem \ref{GPTheorem1} proves a Plemelj formula for the Cauchy transform of an arbitrary finite complex-values Borel measure. In section 3, we introduce the concepts of non-removable boundary and removable set for an arbitrary compact subset $K\subset \C$ and $\mu\in M_0^+(K)$ and prove Theorem \ref{FCharacterization} to characterize them. \cite{cy22} provides a good example of the concepts for a string of beads set. Section 4 proves a key lemma (Lemma \ref{BBFRLambda}) for our approximation scheme. We discuss the relationship between the removable set and $\text{abpe}(\rtkmu)$ in section 5 (Theorem \ref{ABPETheoremRT}). In section 6, we first review the modified Vitushkin approximation scheme of Paramonov (Lemma \ref{BasicEstimate2}). Then, we prove Theorem A and Theorem B.

\section{\textbf{Preliminaries}}

For $\nu \in M_0(\C)$ and $\epsilon > 0,$ $\mathcal C_\epsilon(\nu)$ is defined by
\[
\ \mathcal C_\epsilon(\nu)(z) = \int _{|w-z| > \epsilon}\dfrac{1}{w - z} d\nu (w).
\] 
The (principal value) Cauchy transform
of $\nu$ is defined by
\ \begin{eqnarray}\label{CTDefinition}
\ \mathcal C(\nu)(z) = \lim_{\epsilon \rightarrow 0} \mathcal C_\epsilon(\nu)(z)
\ \end{eqnarray}
for all $z\in\mathbb{C}$ for which the limit exists.
From 
Corollary \ref{ZeroAC} below, we see that \eqref{CTDefinition} is defined for all $z$ except for a set of zero analytic 
capacity.
In particular, it is defined for almost all $z$ with respect to the area measure (the Lebesgue measure on $\C$) $\area .$ Throughout this paper, the Cauchy transform of a measure always means the principal value of the transform.
In the sense of distributions,
 \begin{eqnarray}\label{CTDistributionEq}
 \ \bar \partial \mathcal C(\nu) = - \pi \nu.
 \end{eqnarray}	

The maximal Cauchy transform is defined by
 \[
 \ \mathcal C_*(\nu)(z) = \sup _{\epsilon > 0}| \mathcal C_\epsilon(\nu)(z) |.
 \]

A related capacity, $\gamma _+,$ is defined for subsets $E$ of $\mathbb{C}$ by:
\[
\ \gamma_+(E) = \sup \|\mu \|,
\]
where the supremum is taken over $\mu\in M_0^+(E)$ for which $\|\mathcal{C}(\mu) \|_{L^\infty (\mathbb{C})} \le 1.$ 
Since $\mathcal C\mu$ is analytic in $\mathbb{C}_\infty \setminus \mbox{spt}(\mu)$ and $|(\mathcal{C}(\mu)'(\infty)| = \|\mu \|$, 
we have:
$\gamma _+(E) \le \gamma (E)$.  

X. Tolsa has established the following astounding results. See \cite{Tol03} (also Theorem 6.1 and Corollary 6.3 in \cite{Tol14}) for (1) and (2). See \cite[Proposition 2.1]{Tol02} (also  \cite[Proposition 4.16]{Tol14}) for (3).

\begin{theorem}\label{TolsaTheorem}
(Tolsa 2003)
(1) $\gamma_+$ and $\gamma$ are actually equivalent. 
That is, there is an absolute constant $A_T$ such that for all $E \subset \mathbb{C},$ 
\[
\ \gamma (E) \le A_ T \gamma_+(E).
\] 

(2) Semiadditivity of analytic capacity: for $E_1,E_2,...,E_m \subset \mathbb{C}$ ($m$ may be $\i$),
\[
\ \gamma \left (\bigcup_{i = 1}^m E_i \right ) \le A_T \sum_{i=1}^m \gamma(E_i).
\]

(3) There is an absolute positive constant $C_T$ such that for $a > 0$, we have:  
\[
\ \gamma(\{\mathcal{C}_*(\nu)  \geq a\}) \le \dfrac{C_T}{a} \|\nu \|.
\]

\end{theorem}

For $\eta \in M_0^+(\mathbb C)$, define 
\[
\ N_2(\eta) = \sup_{\epsilon > 0}\sup_{\|f\|_{L^2(\eta)} = 1}\|\mathcal C_\epsilon (f \eta )\|_{L^2(\eta)}.
\]
$\eta$ is $c$-linear growth if
$\eta(\D(\lambda, \delta)) \le c \delta,\text{ for }\lambda\in \C \text{ and }\delta > 0.$ We use $C_1,C_2,...$ for absolute constants that may change from one step to the next.  
The following Proposition is from \cite[Theorem 4.14]{Tol14} and its proofs.

\begin{proposition} \label{GammaPlusThm}
If $F\subset \mathbb C$ is a compact subset and $\eta\in M_0(F),$ then the following properties are true.
\newline
(1) If $\|\CT \eta\|_{L^\i(\C)} \le 1,$ then $\eta$ is $1$-linear growth and $\sup_{\epsilon > 0}\|\mathcal C_\epsilon (\eta )\|_{L^\i(\C)} \le C_1.$
\newline
(2) If $\eta$ is $1$-linear growth and $\|\mathcal C_\epsilon (\eta )\|_{L^\i(\C)} \le 1$ for all $\epsilon > 0,$  then there exists a subset $A\subset F$ such that $\eta (F) \le 2 \eta (A)$ and $N_2(\eta|_A) \le C_2.$
\newline
(3) If $N_2(\eta) \le 1,$ then there exists some function $w$ supported on $F$, with $0\le w \le 1$ such that
$\eta (F) \ \le\  2 \int w d\eta$
 and 
$\sup_{\epsilon > 0}\|\mathcal C _\epsilon (w\eta)\|_{ L^\infty (\mathbb C)} \ \le\  C_3.$
\end{proposition}

Combining Theorem \ref{TolsaTheorem} (1), Proposition \ref{GammaPlusThm}, and \cite{Tol03} (or \cite[Theorem 8.1]{Tol14}), we get the following corollary. The reader may also see \cite[Corollary 3.1]{acy19}.

\begin{corollary}\label{ZeroAC}
If $\nu\in M_0(\C),$ then there exists 
$\mathcal Q \subset \mathbb{C}$ with $\gamma(\mathcal Q) = 0$ such that $\lim_{\epsilon \rightarrow 0}\mathcal{C} _{\epsilon}(\nu)(z)$ 
exists for $z\in\mathbb{C}\setminus \mathcal Q$.
\end{corollary} 

\begin{corollary}\label{ACIncreasing}
Let $E_n\subset E_{n+1}\subset \D(0, R)$ be a sequence of subsets. Then
\[
\ \gamma \left (\cup_{n = 1}^\i E_n \right ) \le C_1 \lim_{n\rightarrow \i} \gamma(E_n).
\]
\end{corollary}

\begin{proof}
By Theorem \ref{TolsaTheorem} (1) and Proposition \ref{GammaPlusThm} (2), there exists a compact subset $F\subset \cup_{n = 1}^\i E_n$ and $\eta \in M_0^+(F)$ with $1$-linear growth such that $N_2(\eta) = 1$ and 
 \[ 
 \ \gamma \left (\cup_{n = 1}^\i E_n \right ) \le C_4 \|\eta\| \le C_4 \lim_{n\rightarrow \i} \eta(E_n). 
 \]
 Since $N_2(\eta|_{E_n}) \le 1,$
 by Proposition \ref{GammaPlusThm} (3) and Theorem \ref{TolsaTheorem} (1), we have $\eta(E_n) \le C_5\gamma(E_n).$
\end{proof}

\begin{lemma} \label{CTMaxFunctFinite}
Let $\{\nu_j\} \subset M_0(\C).$ Then for $\epsilon > 0$, there exists a Borel subset $F$ such that $\gamma (F^c) < \epsilon$ and $\mathcal C_*(\nu_j)(z) \le M_j < \infty$ for $z \in F$.
\end{lemma}

\begin{proof}
Let $A_j = \{\mathcal C_*(\nu_j)(z) \le M_j\}.$ By Theorem \ref{TolsaTheorem} (3), we can select $M_j>0$ so that $\gamma(A_j^c) < \frac{\epsilon}{2^{j+1}A_T}.$ Set $F = \cap_{j=1}^\infty A_j$. Then applying Theorem \ref{TolsaTheorem} (2), we get
\[
 \ \gamma (F^c) \le A_T \sum_{j=1}^\infty \gamma(A_j^c) < \epsilon.
 \]
\end{proof}

Given $0 < \epsilon < \infty$, for $A\subset \mathbb C$, 
we define 
 \[
 \ \mathcal H^1_\epsilon (A) = \inf \left \{\sum_i \text{diam}(A_i):~ A\subset \cup_i A_i,~ \text{diam}(A_i)\le \epsilon \right \}.
 \] 
The $1$-dimensional Hausdorff measure of $A$ is: 
 \[
 \ \mathcal H^1 (A) = \sup_{\epsilon >0} \mathcal H^1_\epsilon (A) = \lim _{\epsilon \rightarrow 0} \mathcal H^1_\epsilon (A). 
 \]
For detailed information about Hausdorff measures, the reader may take a look at the following books: Carleson \cite{car67} and Mattila \cite{mat95}.

For $\nu\in M_0(\C),$ define the zero and non zero linear density sets:
\[
 \ \mathcal{ZD}(\nu) = \{ \lambda:~ \Theta_\nu (\lambda ) = 0 \},~
  \mathcal{ND}(\nu) = \{ \lambda:~ \Theta^*_\nu (\lambda ) > 0 \},
 \]
where $\Theta_\nu (\lambda ) := \lim_{\delta\rightarrow 0} \frac{|\nu |(\D(\lambda , \delta ))}{\delta}$
 if the limit exists and $\Theta_\nu^* (\lambda ) := \underset{\delta\rightarrow 0}{\overline\lim} \frac{|\nu |(\D(\lambda , \delta ))}{\delta}.$
Set
 \[
 \mathcal{ND}(\nu, n) =  \{ \lambda:~ n^{-1} \le \Theta^*_\nu (\lambda ) \le n \} .
 \]
 
 \begin{lemma}\label{RNDecom2} Let $\mu\in M_0^+(\C)$ and $\nu\in M_0(\C).$ The following properties hold.

(a) For $\lambda_2 \ge \lambda_1 > 0$, let  
 $E = \{z:~ \Theta^*_\mu (z) \ge \lambda_1 \}$ and $F = \{z:~ \lambda _1 \le \Theta^*_\mu (z) \le \lambda_2 \}.$
Then $\mathcal H^1(E) \le C_6 \frac{\mu(E)}{\lambda_1}$ and $\mu |_F = g \mathcal H^1|_F$, where $g$ is some Borel function such that $C_7\lambda_1 \le g(z) \le C_8 \lambda_2,~ \mathcal H^1|_F-a.a.$.

(b) If $\mathcal Q_\mu = \{z:~ \Theta_\mu^*(z) = \infty\}$, then $\gamma(\mathcal Q_\mu) = \mathcal H^1(\mathcal Q_\mu) = 0$, 
 \[
 \ \mathcal{ND}(\mu) = \bigcup_n \mathcal{ND}(\mu, n) \cup \mathcal Q_\mu,
 \]
 $\mu |_{\mathcal{ND}(\mu)} = g \mathcal H^1|_{\mathcal{ND}(\mu)}$, and $C_7\frac{1}{n} \le g(z) \le C_8 n,~ \mathcal H^1|_{\mathcal{ND}(\mu, n)}-a.a.$.

(c) If $\eta\in M_0^+(\C)$ is $c$-linear growth and $\eta\perp|\nu|$, then $\eta(\mathcal{ND}(\nu)) = 0$. 
\end{lemma}

\begin{proof}
See \cite[Lemma 8.12]{Tol14} for (a).

(b): From (a),
$\mathcal H^1(\mathcal Q_\mu) \le \mathcal H^1\{z:~ \Theta^*_\mu (z) \ge n\} \le \frac{C_4}{n}\|\mu\|,$ which implies (b). 

(c): By (a), $\mathcal H^1(\mathcal{ND}(\eta,m)\cap \mathcal{ND}(\nu, n)) = 0$ as $\eta\perp|\nu|$. Therefore, $\mathcal H^1(\mathcal{ND}(\eta)\cap \mathcal{ND}(\nu)) = 0$. On the other hand, $\mathcal H^1(\mathcal{ND}(\nu, n)) < \infty$, which implies $\eta |_{\mathcal{ND}(\nu)}$ is absolutely continuous with respect to $\mathcal H^1  |_{\mathcal{ND}(\nu)}$ since $\eta$ is linear growth. Hence, $\eta(\mathcal{ZD}(\eta) \cap \mathcal{ND}(\nu)) = 0$, which implies $\eta(\mathcal{ND}(\nu)) = 0$.
\end{proof}

\begin{definition}\label{GLDef}
Let  $\mathcal Q$ be a set with $\gamma(\mathcal Q) = 0$.
Let $f(z)$ be a function defined 
on $\D(\lambda, \delta_0)\setminus \mathcal Q$ for some $\delta_0 > 0.$ The function $f$ has a $\gamma$-limit $a$ at $\lambda$ if
\[  
 \  \lim_{\delta \rightarrow 0} \dfrac{\gamma(\D(\lambda, \delta) \cap \{|f(z) - a| > \epsilon\})} {\delta}= 0
\]
for all $\epsilon > 0$. If in addition, $f(\lambda)$ is well defined and $a = f(\lambda)$, then $f$ is $\gamma$-continuous at $\lambda$. 
\end{definition}
 
 The following lemma is from \cite[Lemma 3.2]{acy19}. 

\begin{lemma}\label{CauchyTLemma} 
Let $\nu\in M_0(\mathbb{C})$ and assume that for some $\lambda _0\in\mathbb C,$ 
$\Theta_\nu (\lambda _0) = 0$ and $\mathcal{C} (\nu)(\lambda _0) = \lim_{\epsilon \rightarrow 0}\mathcal{C} _{\epsilon}(\nu)(\lambda _0)$ exists.
Then the  Cauchy transform $\mathcal{C}(\nu)(\lambda)$ is $\gamma$-continuous at $\lambda_0$.
\end{lemma}

Let $\mathbb R$ be the real line. Let $A:\mathbb R\rightarrow \mathbb R$ be a Lipschitz function. The Lipschitz graph $\Gamma$ of $A$ is defined by 
 $\Gamma = \left \{(x,A(x)):~x\in \mathbb R \right \}$.
We define, for $\lambda \in \Gamma,$ the open upper cone 
\[
	 \ UC (\lambda, \alpha ) = \left \{z \in \mathbb C :~ |Re(z) -  Re(\lambda )| < \alpha (Im(z) -  Im(\lambda ))\right \}
 \]
	and the open lower cone 
\[
	 \ LC (\lambda, \alpha ) = \left \{z \in \mathbb C :~ |Re(z) -  Re(\lambda )| < - \alpha (Im(z) -  Im(\lambda ))\right \}.
 \]
	Set $UC (\lambda, \alpha, \delta ) = UC (\lambda, \alpha)\cap \D(\lambda ,  \delta)$	and
	$LC (\lambda, \alpha, \delta ) = LC (\lambda, \alpha)\cap \D(\lambda ,  \delta)$.
		
If $\alpha < \frac{1}{\|A'\|_\infty},$ then, for almost all $\lambda\in \Gamma,$ there exists $\delta > 0$ such that 
	\[
	 \ UC (\lambda, \alpha, \delta) \subset U_\Gamma:=\left\{z\in \mathbb C:~ Im(z) > A(Re(z))\right\}
	 \]
	 and 
	 \[
	 \ LC (\lambda, \alpha, \delta)\subset L_\Gamma:=\left\{z\in \mathbb C:~ Im(z) < A(Re(z))\right\}.
	 \]
	We consider the usual complex-valued measure
	 \[
	 \ \dfrac{1}{2\pi i} dz_{\Gamma} = \dfrac{1 + iA'(Re(z))}{2\pi i (1 + A'(Re(z))^2)^{\frac{1}{2}}} d\mathcal H^1 |_{\Gamma} = L(z)d\mathcal H^1 |_{\Gamma}.
	 \]
Notice that
$|L(z)| = \frac{1}{2\pi}.$
Plemelj's formula for Lipschitz graphs as in \cite[Theorem 8.8]{Tol14} is the following.
	
	\begin{lemma}\label{PFLipschitz}
	Let $A : \mathbb R \rightarrow \mathbb R$ be a Lipschitz function with its graph $\Gamma$. Let $0 < \alpha < \frac{1}{\|A'\|_\infty}$. Then, for all $f\in L^1(\mathcal H^1 |_{\Gamma})$, the non-tangential limits of $\mathcal C(f dz_{\Gamma})(\lambda )$ exists for $\mathcal H^1 |_{\Gamma}-a.a.$ and the following identities hold $\mathcal H^1 |_{\Gamma}-a.a.$ for $\lambda\in \Gamma$:
	\[
	 \ \dfrac{1}{2\pi i} \lim_{z\in UC (\lambda, \alpha ) \rightarrow \lambda}\mathcal C(f dz_{\Gamma})(z) = \dfrac{1}{2\pi i}\mathcal C(f dz_{\Gamma})(\lambda) + \dfrac{1}{2} f(\lambda),
	\]
	and
	\[
	 \ \dfrac{1}{2\pi i} \lim_{z\in LC (\lambda, \alpha ) \rightarrow \lambda}\mathcal C(f dz_{\Gamma})(z) = \dfrac{1}{2\pi i}\mathcal C(f dz_{\Gamma})(\lambda) - \dfrac{1}{2} f(\lambda).
	\]
\end{lemma}

\begin{lemma} \label{LGCTL2Bounded}
 Let $A:\mathbb R\rightarrow \mathbb R$ be a Lipschitz function with graph $\Gamma$. 
 Then 
 $N_2(\mathcal H^1 |_\Gamma) < \i,$
 while $N_2(\mathcal H^1 |_\Gamma)$ only depends on $\|A'\|_\infty$.
Consequently, there is a constant $c_\Gamma > 0$ that depends only on $\|A'\|_\infty$ such that for $E\subset \Gamma$,
 \begin{eqnarray}\label{HACEq}
 \ c_\Gamma \mathcal H^1 |_\Gamma (E) \le \gamma (E) \le \mathcal H^1 |_\Gamma (E). 
 \end{eqnarray}	
\end{lemma}

\begin{proof}  See \cite{cmm82} or \cite[Theorem 3.11]{Tol14} for $N_2(\mathcal H^1 |_\Gamma) < \i.$ Therefore, from Proposition \ref{GammaPlusThm} (3) and Theorem \ref{TolsaTheorem} (1), we conclude that $c_\Gamma \mathcal H^1 |_\Gamma (E) \le \gamma (E)$ holds. See \cite[Theorem 1.21]{Tol14} for $\gamma (E) \le \mathcal H^1 |_\Gamma (E).$
\end{proof}

We now improve Lemma \ref{PFLipschitz} with the following theorem.

\begin{theorem}\label{GPTheorem1}
(Plemelj's Formula for an arbitrary measure) Let $\alpha < \frac{1}{\|A'\|_\infty}$ and $\nu\in M_0(\C)$. Suppose that $\nu = h\mathcal H^1 |_{\Gamma} + \nu_s$ is the Radon-Nikodym 
decomposition with respect to $\mathcal H^1 |_{\Gamma}$, where $h\in L^1(\mathcal H^1 |_{\Gamma})$ and $\nu_s\perp \mathcal H^1 |_{\Gamma}$. Then there exists a subset $\mathcal Q\subset \mathbb C$ with $\gamma(\mathcal Q) = 0$, such that the following hold:

(a) $\mathcal C(\nu ) (\lambda_0) = \lim_{\epsilon\rightarrow 0} \mathcal C_{\epsilon}(\nu)(\lambda_0)$ exists for $\lambda_0\in \mathbb C\setminus \mathcal Q$.

(b) For $\lambda_0 \in \Gamma \setminus \mathcal Q$ and $\epsilon > 0$, $v^+(\nu, \Gamma)( \lambda_0) := \mathcal C(\nu ) (\lambda_0) + \frac{h(\lambda_0)}{2L(\lambda_0)},$
\[
 \ \lim_{\delta \rightarrow 0} \dfrac{\gamma(U_\Gamma \cap \D(\lambda, \delta ) \cap \{|\mathcal{C}(\nu )(\lambda ) - v^+(\nu, \Gamma)( \lambda_0)| > \epsilon\})} {\delta}= 0.  
 \]
 
(c) For $\lambda_0 \in \Gamma \setminus \mathcal Q$ and $\epsilon > 0$, $v^-(\nu, \Gamma)( \lambda_0) := \mathcal C(\nu ) (\lambda_0) - \frac{h(\lambda_0)}{2L(\lambda_0)},$
\[
 \ \lim_{\delta \rightarrow 0} \dfrac{\gamma(L_\Gamma \cap \D(\lambda, \delta ) \cap \{|\mathcal{C}(\nu )(\lambda ) - v^-(\nu, \Gamma)( \lambda_0)| > \epsilon\})} {\delta}= 0.  
 \]

(d) For $\lambda_0 \in \Gamma \setminus \mathcal Q$ and $\epsilon > 0$, $v^0(\nu, \Gamma)( \lambda_0) := \mathcal C(\nu ) (\lambda_0),$
\[
 \ \lim_{\delta \rightarrow 0} \dfrac{\gamma(\Gamma \cap \D(\lambda, \delta ) \cap \{|\mathcal{C}(\nu )(\lambda ) - v^0(\nu, \Gamma)( \lambda_0)| > \epsilon\})} {\delta}= 0.  
 \]
\end{theorem}


As $\nu$ is compactly supported, 
 We will just consider the portion  of the graph of $\Gamma$ that lies in $\{z:~ |Re(z)| < r\}$ for some $r > 0.$
Since $\nu_s\perp \mathcal H^1|_{\Gamma}$, using Lemma \ref{RNDecom2} (c) and Lemma \ref{CauchyTLemma}, we conclude that (b), (c), and (d) holds for the measure $\nu_s$. So we may assume that $\nu = hd\mathcal H^1 |_{\Gamma}$.

\begin{proof}
(Theorem \ref{GPTheorem1} (d)): Since $\gamma |_{\Gamma} \approx \mathcal H^1 |_{\Gamma}$ by \eqref{HACEq}, we use $\mathcal H^1 |_{\Gamma}$ instead of $\gamma$ in the proof.

From Lemma \ref{LGCTL2Bounded}, we see  that 
the Cauchy transform of $\mathcal H^1 |_{\Gamma}$ is bounded on $L^2(\mathcal H^1 |_{\Gamma})$. So we get $\mathcal C\mathcal H^1 |_{\Gamma}\in L^2(\mathcal H^1 |_{\Gamma})$ as $\Gamma \subset \{z:~ |Re(z)| < r\}$. Also $h\in L^1(\mathcal H^1 |_{\Gamma})$. Let $\lambda _0$ be a Lebesgue point of $h$ and $\mathcal C\mathcal H^1 |_{\Gamma}$. That is,
\begin{eqnarray}\label{LPoint1}
 \ \lim_{\delta\rightarrow 0} \dfrac{1}{\delta}\int_{\D(\lambda_0,\delta)} |h(z) - h(\lambda_0)|d\mathcal H^1 |_{\Gamma}(z) = 0 
\end{eqnarray}
and 
\begin{eqnarray}\label{LPoint2}
 \ \lim_{\delta\rightarrow 0} \dfrac{1}{\delta}\int_{\D(\lambda_0,\delta)} |\mathcal C\mathcal H^1 |_{\Gamma} (z) - \mathcal C\mathcal H^1 |_{\Gamma} (\lambda_0)|d\mathcal H^1 |_{\Gamma}(z) = 0 
\end{eqnarray}
as $\mathcal H^1 |_{\Gamma}(\D(\lambda_0,\delta)) \approx \delta.$  
 
For $\epsilon > 0$, we get (assuming $h(\lambda_0)\ne 0$):
 \[
 \ \begin{aligned}
 \ & \left \{|\mathcal C\nu (\lambda) - \mathcal C\nu (\lambda_0) | > \epsilon \right \}\\
 \ \subset & \left \{|\mathcal C((h-h(\lambda_0))\mathcal H^1 |_\Gamma )(\lambda) - \mathcal C((h-h(\lambda_0))\mathcal H^1 |_\Gamma )(\lambda_0) |  > \frac{\epsilon}{2}\right\} \\
 \ &\bigcup \left \{|\mathcal C\mathcal H^1 |_{\Gamma} (\lambda) - \mathcal C\mathcal H^1 |_{\Gamma} (\lambda_0) | > \frac{\epsilon}{2|h(\lambda_0)|} \right \}.
 \ \end{aligned}
 \]
It follows from  \eqref{LPoint2} that 
 \[
 \ \lim_{\delta\rightarrow 0} \dfrac{\mathcal H^1 |_{\Gamma} \left ( \left \{|\mathcal C\mathcal H^1 |_{\Gamma} (\lambda) - \mathcal C\mathcal H^1 |_{\Gamma} (\lambda_0) | > \dfrac{\epsilon}{2|h(\lambda_0)|} \right \}\cap \D(\lambda_0, \delta)\right)}{\delta} = 0. 
 \]
Therefore, we may assume that $h(\lambda_0) = 0$. In this case, $\nu$ satisfies the assumptions of Lemma \ref{CauchyTLemma} as \eqref{LPoint1} holds and we 
get 
(d) because $\mathcal C(\nu)(z)$ is $\gamma$-continuous at $\lambda_0$ by Lemma \ref{CauchyTLemma}.
\end{proof}

For a complex-valued function $f$ defined on $\C,$ recall that 
\[
\ \mathcal N(f) = \{\lambda: ~ f(\lambda) \ne 0 \},~\mathcal Z(f)  = \{\lambda: ~ f(\lambda) = 0 \}
\]

\begin{proof}
(Theorem \ref{GPTheorem1} (b)):
For $a>0$, let $A_n$ be the set of $\lambda\in \mathcal N(h)$
satisfying 
 \[
 \ \left |\mathcal C (\nu)(z) - \mathcal C (\nu)(\lambda) - \frac{h(\lambda)}{2L(\lambda)} \right | < \dfrac{a}{2}
 \]
for $z\in UC(\lambda, \alpha, \frac{1}{n})$.
 Then from Lemma \ref{PFLipschitz}, we obtain
 \begin{eqnarray}\label{ANEq}
 \ \mathcal H^1 |_{\Gamma} \left (\mathcal N(h) \setminus \cup_{n=1}^\infty A_n \right ) = 0.
 \end{eqnarray}
It is clear that there exists $B_n \subset A_n$ with $\mathcal H^1(B_n) = 0$ such that for each $\lambda_0\in A_n\setminus B_n$, the following two properties hold.
\newline
(1) $\lambda_0$ is a Lebesgue point of $\chi_{A_n}$ and $\frac{h(\lambda)}{L(\lambda)},$ where $\chi_B$ is the characteristic function of the set $B,$ 
that is,
 \[
 \ \lim_{\delta\rightarrow 0} \dfrac{\mathcal H^1 |_{\Gamma}(A_n^c \cap \D(\lambda_0, \delta))}{\delta} = 0
 \]
and
 \[
 \ \lim_{\delta\rightarrow 0 } \dfrac{\mathcal H^1 |_{\Gamma} \left ( \left \{\lambda : \left |\frac{h(\lambda)}{L(\lambda)} - \frac{h(\lambda_0)}{L(\lambda_0)} \right | \ge \frac{a}{2} \right  \} \cap \D(\lambda _0, \delta) \right)}{\delta} = 0;
 \]
 \newline
(2) $\lambda_0$ satisfies (d), that is, 
 \[
 \ \lim_{\delta\rightarrow 0 } \dfrac{\mathcal H^1 |_{\Gamma} \left ( \{\lambda : |\mathcal C(\nu)(\lambda) - \mathcal{C}(\nu)(\lambda _0)| \ge \frac{a}{4} \} \cap \D(\lambda _0, \delta) \right)}{\delta} = 0.
 \]
 
Therefore,
 \[
 \ \lim_{\delta\rightarrow 0 } \dfrac{\mathcal H^1 |_{\Gamma} \left ( (A_n^c \cup \{\lambda : |v^+(\nu, \Gamma)( \lambda) - v^+(\nu, \Gamma)( \lambda_0)| \ge \frac{a}{2} \}) \cap \D(\lambda _0, \delta) \right)}{\delta} = 0.
 \]
For $\epsilon > 0$, there exists $\delta_0 > 0$ ($\delta_0 < \frac{1}{2n}$) and an open set $O_\delta \subset \D(\lambda _0, \delta)$ for $\delta < \delta_0$ such that 
 \[
 \ \left (A_n^c \cup \left \{\lambda : |v^+(\nu, \Gamma)( \lambda) - v^+(\nu, \Gamma)( \lambda_0)| \ge \frac{a}{2} \right \} \right ) \cap \D(\lambda _0, \delta) \subset O_\delta
 \]
and $\mathcal H^1 |_{\Gamma} (O_\delta) < \epsilon\delta$. Let $E = \Gamma \cap (\overline{\D(\lambda_0,\delta)}\setminus O_\delta)$. Denote
 \[
 \ B_\delta = \bigcup_{\lambda\in E} \overline{UC \left (\lambda, \alpha,\frac{1}{n}\right )}.
 \]
It is easy to check that $B_\delta$ is closed. By construction, we see that
 \[
 \ |\mathcal C (\nu)(z) - v^+(\nu, \Gamma)( \lambda_0)| < a
 \]
for $z\in B_\delta\setminus \Gamma$. Write 
 \[
 \ U_\Gamma \cap \D(\lambda_0, \delta) \setminus B_\delta = \cup _{k=1}^\infty G_k
 \]
where $G_k$ is a connected component. It is clear that $diam(G_k) \approx \mathcal H^1 |_{\Gamma} (\overline{G_k}\cap\Gamma)$. Let $I_k$ be the interior of $\overline{G_k}\cap\Gamma$ on $\Gamma$. Then $\mathcal H^1 |_{\Gamma} (\overline{G_k}\cap\Gamma) = \mathcal H^1 |_{\Gamma} (I_k)$ and $\cup_k I_k \subset O_\delta\cap \Gamma$. Hence, from Theorem \ref{TolsaTheorem} (2), we get 
 \[
 \ \gamma ( U_\Gamma \cap \D(\lambda_0, \delta) \setminus B_\delta ) \le A_T\sum_{k=1}^\infty \gamma (G_k) \le A_T\sum_{k=1}^\infty diam (G_k) \le A_TC_9\epsilon\delta.
 \]

This implies
 \begin{eqnarray}\label{ANEq1}
 \ \lim_{\delta\rightarrow 0 } \dfrac{\gamma \left ( \{\lambda : |\mathcal C (\nu)(\lambda) - v^+(\nu, \Gamma)( \lambda_0)| > a \} \cap \D(\lambda _0, \delta)\cap U_\Gamma \right)}{\delta} = 0.
 \end{eqnarray}
So from \eqref{ANEq}, we have proved that there exists $\mathcal Q_a$ with $\mathcal H^1 (\mathcal Q_a) = 0$ such that for each $\lambda_0\in \mathcal N(h) \setminus \mathcal Q_a$, \eqref{ANEq1} holds. Set $\mathcal Q_0 = \cup_{m=1}^\infty \mathcal Q_{\frac{1}{m}}$. Then \eqref{ANEq1} holds for $\lambda_0\in \mathcal N(h) \setminus \mathcal Q_0$ and all $a > 0$. This proves (b).
\end{proof}

The proof of (c) is the same as (b).
We point out if $\Gamma$ is a rotation of a Lipschitz graph (with rotation angle $\beta$ and Lipschitz function $A$) at $\lambda_0$, then
 \begin{eqnarray}\label{VPlusBeta}
 \ v^+(\nu, \Gamma, \beta)( \lambda) := \mathcal C(\nu ) (\lambda) + \frac{e^{-i\beta}h(\lambda)}{2L((\lambda - \lambda_0)e^{-i\beta} + \lambda_0)}.
 \end{eqnarray}
Similarly,
 \begin{eqnarray}\label{VMinusBeta}
 \ v^-(\nu, \Gamma, \beta)( \lambda) := \mathcal C(\nu ) (\lambda) - \frac{e^{-i\beta}h(\lambda)}{2L((\lambda - \lambda_0)e^{-i\beta} + \lambda_0)}.
 \end{eqnarray}
 Hence, $v^+(\nu, \Gamma)( \lambda) = v^+(\nu, \Gamma, 0)( \lambda)$ and $v^-(\nu, \Gamma)( \lambda) = v^-(\nu, \Gamma, 0)( \lambda).$

 \section{\textbf{Definition of  the non-removable boundary and the removable set}}
  
From now on, we fix a compact subset $K\subset \C,$ $\mu\in M_0^+(K),$ $1 \le t < \i,$ $\frac 1t + \frac 1s = 1,$ and 
\[
\ \Lambda = \{g_j\} \subset R^t(K,\mu)^\perp := \left \{g\in L^s(\mu): ~ \int fg d\mu = 0 \text{ for } f\in \text{Rat}(K)  \right \}
\]
is a $L^s(\mu)$ norm dense subset.
 
 The following definition generalizes \cite[Definition 4.3]{cy22}.
 
 \begin{definition}\label{R0ForRILambda}
 The non-removable boundary of zero density for $\Lambda$ is defined by
 \[
 \ \mathcal F_0 (\Lambda) = \bigcap_{j = 1}^\infty \{z\in \mathcal {ZD}(\mu):~ \mathcal C(g_j\mu)(z) = 0\}
 \]
 and the removable set of zero density for $\Lambda$ is defined by
 \[
 \ \mathcal R_0 (\Lambda) = \bigcup_{j = 1}^\infty \{z\in \mathcal {ZD}(\mu):~ \mathcal C(g_j\mu)(z) \ne 0\}.
 \]
 \end{definition}
 For the non-zero density part of $\mu,$ we need to introduce further notation. Let $\Gamma_0$ be a 
  Lipschitz graph with rotation angle $\beta_0,$ define
 \[
 \ \mathcal Z_+(\Lambda, \Gamma_0) = \bigcap_{j = 1}^\infty \{z\in \Gamma_0: ~ v^+(g_j\mu, \Gamma_0, \beta_0)(z) = 0 \}, 
 \]
 \[
 \ \mathcal Z_-(\Lambda, \Gamma_0) = \bigcap_{j = 1}^\infty \{z\in \Gamma_0: ~ v^-(g_j\mu, \Gamma_0, \beta_0)(z) = 0 \},
 \]
 \[
 \ \mathcal N_+(\Lambda, \Gamma_0) = \bigcup_{j = 1}^\infty \{z\in \Gamma_0: ~ v^+(g_j\mu, \Gamma_0, \beta_0)(z) \ne 0 \},
 \]
 and 
 \[
 \ \mathcal N_-(\Lambda, \Gamma_0) = \bigcup_{j = 1}^\infty \{z\in \Gamma_0: ~ v^-(g_j\mu, \Gamma_0, \beta_0) \ne 0 \}.
 \]
 The following lemma, which describes the decomposition of $\mu\in M_0^+(K),$ is important for our definitions.
 
 \begin{lemma}\label{GammaExist}
Let $K$ be a compact subset and $\mu\in M_0^+(K).$ Then there is a sequence of Lipschitz functions $A_n: \mathbb R\rightarrow \mathbb R$ and 
such that their  graphs $\Gamma_n$ with rotation angles $\beta_n$ satisfy
the following properties:

(1) Let $\Gamma = \cup_n \Gamma_n$. Then
$\mu = h\mathcal H^1 |_{\Gamma} + \mu_s$ is the Radon-Nikodym decomposition with respect to $\mathcal H^1 |_{\Gamma}$, where $h\in L^1(\mathcal H^1 |_{\Gamma})$ and $\mu_s\perp \mathcal H^1 |_{\Gamma};$

(2) $\mathcal {ND}(\mu ) \approx \mathcal N (h), ~\gamma-a.a.;$

(3) $\Theta_{\mu_s}(z) = 0, ~\gamma-a.a..$ 
\end{lemma}

\begin{proof} From Lemma \ref{RNDecom2}
(b), we have
 \[
 \ \mathcal{ND}(\mu) = \bigcup_n \mathcal{ND}(\mu, n) \cup \mathcal Q_\mu.
 \] 
Thus, we find $h$ such that $\mu |_{\mathcal{ND}(\mu, n)} = h \mathcal H^1 |_{\mathcal{ND}(\mu, n)}$ and $\mathcal H^1(\mathcal{ND}(\mu, n)) < \i.$ 
	
From \cite[Theorem 1.26]{Tol14}, $\mathcal{ND}(\mu, n) = E_r\cup E_u,$ where $E_r$ is rectifiable and $E_u$ is purely unrectifiable. From David's Theorem (see \cite{dav98} or \cite[Theorem 7.2]{Tol14}), we see that $\gamma(E_u) = 0.$
Applying \cite[Proposition 4.13]{Tol14}, we find a sequence of (rotated) Lipschitz graphs $\{\Gamma_{nm}\}$ and $E_1$ with $\mathcal H^1(E_1) = 0$ such that
$E_r \subset E_1 \cup \cup_{m=1}^\infty \Gamma_{nm}.$ 
Let $\mathcal Q = E_u\cup E_1.$ Clearly, $\gamma(\mathcal Q) = 0.$ Hence, there exists a sequence of $\{\Gamma_n\}$ such that (1), (2), and (3) hold.
\end{proof}  

We are now ready to define the non-removable boundary and removable set for non-zero density portion of $\mu,$ which extends \cite[Definition 4.6]{cy22}.

\begin{definition}\label{FPMR1Def} 
Let $\Gamma_n$ and $h$ be as in Lemma \ref{GammaExist}. Define the upper non-removable boundary for the non-zero density portion of $\mu$ as
\[
 \ \mathcal F_+ (\Lambda) :=  \bigcup_{n = 1}^\infty \mathcal Z_+(\Lambda, \Gamma_n) \cap \mathcal{ND}(\mu),
 \]
 the lower non-removable boundary for  the non-zero density portion of $\mu$ as
 \[
 \ \mathcal F_-(\Lambda) :=  \bigcup_{n = 1}^\infty \mathcal Z_-(\Lambda, \Gamma_n) \cap \mathcal{ND}(\mu),
 \]
 and the removable set for the non-zero density portion of $\mu$ as
 \[
 \ \mathcal R_1(\Lambda) :=  \bigcup_{n = 1}^\infty \left ( \mathcal N_+(\Lambda, \Gamma_n)\cap \mathcal N_-(\Lambda, \Gamma_n) \right ) \cap \mathcal {ND} (\mu),
 \]
 \end{definition}
 
 Finally, we have the following definition.
 
 \begin{definition}\label{FRDef}
 The non-removable boundary for $\Lambda$ is defined as
 \[
 \ \mathcal F(\Lambda) = \mathcal F_0(\Lambda)\cup \mathcal F_+(\Lambda) \cup \mathcal F_-(\Lambda)
 \]
and the removable set for $\Lambda$ is defined as
 \[
 \ \mathcal R(\Lambda) :=  \mathcal R_0(\Lambda) \cup \mathcal R_1(\Lambda).
 \]
\end{definition}

In the remaining section, we discuss a characterization of $\mathcal F(\Lambda)$ and $\mathcal R(\Lambda)$ which implies Definition \ref{FRDef} does not depend on the choices of $\{\Gamma_n\}$, up to a set of zero analytic capacity. First we need the following lemma that will also be used later. 
 
 \begin{lemma} \label{BBFunctLemma}
Suppose that $\{g_n\}\subset L^1(\mu)$ and $E$ is a compact subset with $\gamma(E) > 0$. Then there exists $\eta \in M_0^+(E)$ satisfying:

(1) $\eta$ is of $1$-linear growth, $\|\mathcal C_{\epsilon}(\eta)\|_{L^\infty (\mathbb C)}\le 1$ for all $\epsilon > 0,$ and 
$\gamma(E) \le C_{10} \|\eta\|$;

(2) $\mathcal C_*(g_n\mu )\in L^\infty(\eta)$;

(3) there exists a subsequence $f_k(z) = \mathcal C_{\epsilon_k}(\eta)(z)$ such that $f_k$ converges to $f\in L^\infty(\mu)$ in weak-star topology, and $f_k(\lambda) $ converges to $f(\lambda) = \mathcal C(\eta)(\lambda)$ uniformly on any compact subset of $E^c$ as $\epsilon_k\rightarrow 0$.
Moreover for $n \ge 1,$ 
\begin{eqnarray}\label{BBFunctLemmaEq1}
 \ \int f(z) g_n(z)d\mu (z) = - \int \mathcal C(g_n\mu) (z) d\eta (z),
 \end{eqnarray}
and for $\lambda\in E^c$,
 \begin{eqnarray}\label{BBFunctLemmaEq2}
 \ \int \dfrac{f(z) - f(\lambda)}{z - \lambda} g_n(z)d\mu (z) = - \int \mathcal C(g_n\mu) (z) \dfrac{d\eta (z)}{z - \lambda}.
 \end{eqnarray} 
\end{lemma}

\begin{proof}
From Lemma \ref{CTMaxFunctFinite}, we find $E_1\subset E$ such that $\gamma(E\setminus E_1) < \frac{\gamma(E)}{2A_T}$ and $\mathcal C_*(g_n\mu )(z) \le M_n < \infty$ for $z\in E_1$. Using Theorem \ref{TolsaTheorem} (2), we get
 $\gamma(E_1) \ge \frac{1}{A_T}\gamma(E) - \gamma(E\setminus E_1) \ge \frac{1}{2A_T}\gamma(E).$
Using Theorem \ref{TolsaTheorem} (1) and Proposition \ref{GammaPlusThm} (1),
there exists $\eta\in M_0^+(E_1)$ satisfying (1). So (2) holds. Moreover, using Proposition \ref{GammaPlusThm} again, we may assume $\eta = w\eta_1,$ where $\eta_1$ is of $1$-linear growth, $0\le w \le 1,$ and $N_2(\eta_1) \le 1.$ Hence, from Theorem \ref{TolsaTheorem} (1) and Proposition \ref{GammaPlusThm} (3), we conclude a $\gamma$ zero set is also a zero set of $\eta.$
Clearly,
 \begin{eqnarray}\label{lemmaBasicEq3}
 \  \int \mathcal C_\epsilon(\eta)(z) g_nd\mu  = - \int \mathcal C_\epsilon(g_n\mu)(z) d\eta
 \end{eqnarray}
for $n \ge 1$. We can choose a sequence $f_k(\lambda) = \mathcal C_{\epsilon_k}(w\eta)(\lambda)$ that converges to $f$ in $L^\infty(\mu)$ weak-star topology and $f_k(\lambda)$ uniformly tends to $f(\lambda)$ on any compact subset of $E^c$. On the other hand, $|\mathcal C_{\epsilon_k}(g_n\mu)(z) | \le M _n,~ \eta |_E-a.a.$ and by Corollary \ref{ZeroAC}, $\lim_{k\rightarrow \infty} \mathcal C_{\epsilon_k}(g_n\mu)(z)  = \mathcal C(g_n\mu)(z) ,~ \eta -a.a.$. We apply the Lebesgue dominated convergence theorem to the right hand side of \eqref{lemmaBasicEq3} and get \eqref{BBFunctLemmaEq1} for $n \ge 1$. For \eqref{BBFunctLemmaEq2}, let $\lambda\notin E$ and $d = \text{dist}(\lambda, E)$,
for $z\in \D(\lambda, \frac{d}{2})$ and $\epsilon < \frac{d}{2}$, we have
 \[
 \ \left |\dfrac{\mathcal C_\epsilon (\eta)(z) - f(\lambda)}{z - \lambda} \right |\le \left |\mathcal C_\epsilon  \left (\dfrac{\eta (s)}{s - \lambda} \right ) (z) \right | \le 
\dfrac{2}{d^2}\|\eta\|. 
 \]
For $z\notin \D(\lambda, \frac{d}{2})$ and $\epsilon < \frac{d}{2}$,
 \[
 \ \left |\dfrac{\mathcal C_\epsilon (\eta)(z) - f(\lambda)}{z - \lambda} \right | \le \dfrac{4}{d}.
 \]
Thus, we can replace the above proof for the measure $\frac{\eta (s)}{s - \lambda}$. In fact, we can choose a subsequence  $\{\mathcal C_{\epsilon_{k_j}} (\eta)\}$ such that $e_{k_j}(z) = \frac{\mathcal C_{\epsilon_{k_j}} (\eta)(z) - f(\lambda)}{z - \lambda}$ converges to $e(z)$ in weak-star topology. Clearly, $(z-\lambda)e_{k_j}(z)  + f(\lambda) = \mathcal C_{\epsilon_{k_j}} (\eta)(z)$ converges to $(z-\lambda)e(z)  + f(\lambda) = f(z)$ in weak-star topology.  
On the other hand, \ref{lemmaBasicEq3} becomes
 \begin{eqnarray}\label{lemmaBasicEq31}
 \  \int \mathcal C_{\epsilon_{k_j}}(\dfrac{\eta(s)}{s-\lambda})(z) g_nd\mu  = - \int \mathcal C_{\epsilon_{k_j}}(g_n\mu)(z) \dfrac{d\eta(z)}{z-\lambda}
 \end{eqnarray}
and for $\epsilon_{k_j} < \frac{d}{2}$, we have
 \begin{eqnarray}\label{lemmaBasicEq32}
 \left | \mathcal C_{\epsilon_{k_j}}(\dfrac{\eta(s)}{s-\lambda})(z) - e_{k_j}(z) \right |
\ \le \begin{cases}0, & z\in \D(\lambda, \frac{d}{2}), \\ \dfrac{2}{d^2} \eta(\D(z, \epsilon_{k_j})), & z\notin \D(\lambda, \frac{d}{2}), \end{cases}
\end{eqnarray}
which goes to zero as $\epsilon_{k_j} \rightarrow 0$. Combining \eqref{lemmaBasicEq31}, \eqref{lemmaBasicEq32}, and Lebesgue dominated convergence theorem, we prove the equation \eqref{BBFunctLemmaEq2}. (3) is proved.    
\end{proof}

For a Lipschitz graph $\Gamma_0$ with rotation angle $\beta_0,$ define
 \[
 \ \mathcal N_0(\Lambda, \Gamma_0) = \bigcup_{j = 1}^\infty \{z\in \Gamma_0: ~ \CT (g_j\mu)(z) \ne 0 \}.
 \]
 
 \begin{lemma} \label{NPMLemma}
 Let $\Gamma_n$ and $h$ be as in Lemma \ref{GammaExist}. Then, for $n \ge 1,$
 \[
 \ \mathcal N_+(\Lambda, \Gamma_n) \cup \mathcal N_-(\Lambda, \Gamma_n) \subset \mathcal N_0(\Lambda, \Gamma_n),~\mathcal H^1|_{\Gamma_n}-a.a..
 \]	
 \end{lemma}
 
 \begin{proof}
 Without loss of generality, we assume $n=1$ and $\beta_1=0.$
 Suppose there exists a compact subset $E \subset \mathcal N_+(\Lambda, \Gamma_1) \cup \mathcal N_-(\Lambda, \Gamma_1) \setminus \mathcal N_0(\Lambda, \Gamma_1)$ such that $\mathcal H^1(E) > 0.$ Then $\CT (g_j\mu)(z) = 0,~\mathcal H^1|_E-a.a.$ and 
 \begin{eqnarray}\label{NPMLemmaEq1}
 \ v^+(g_j\mu,\Gamma_1)(z) = \dfrac{(g_jh)(z)}{2L(z)},~ v^-(g_j\mu,\Gamma_1)(z) = - \dfrac{(g_jh)(z)}{2L(z)},~\mathcal H^1|_E-a.a.
 \end{eqnarray}
 for all $j\ge 1.$ Let $\eta	$ and $f$ be from Lemma \ref{BBFunctLemma} for $\Lambda$ and $E$ as $\gamma(E) > 0$ by \eqref{HACEq}. We may assume $\eta = w\mathcal H^1|_E,$ where $0\le w(z) \le 1$ on $E,$ since $\eta$ is of $1$-linear growth. From Lemma \ref{BBFunctLemma} (3), we see $f\in R^{t,\i}(K,\mu)$ as $\Lambda$ is dense in $R^t(K,\mu)^\perp$ and
 \begin{eqnarray}\label{NPMLemmaEq2}
 \ \CT\eta (\lambda)\CT (g_j\mu)(\lambda) = \CT (fg_j\mu)(\lambda), ~\gamma|_{E^c}-a.a.,~\text{ for } j\ge 1. 
 \end{eqnarray}
 Let $\{r_n\}\subset \text{Rat}(K)$ such that $\|r_n - f\|_{L^t(\mu)}\rightarrow 0.$ Clearly, 
 \[
 \ \mathcal C(r_ng_j\mu)(z) = r_n (z)\mathcal C(g_j\mu)(z) = 0,~\eta-a.a.
 \]
 since a $\gamma|_E$ zero set is also a zero set of $\eta$ by \eqref{HACEq}. By Theorem \ref{TolsaTheorem} (3), we get
 \[
 \begin{aligned}
 \ \eta\{|\mathcal C(fg_j\mu)| > \epsilon\} = &\eta\{|\mathcal C((r_n-f)g_j\mu)| > \epsilon\} \\
 \ \le & \eta\{\mathcal C_*((r_n-f)g_j\mu) > \epsilon\} \\
 \ \le & \dfrac{C_{11}\|r_n-f)\|\|g_j\|}{\epsilon} \rightarrow 0 \text{ as } n \rightarrow 0.
 \end{aligned}
 \]
 Therefore, $\mathcal C(fg_j\mu)(z) = 0,~\eta-a.a.$ and we have
 \begin{eqnarray}\label{NPMLemmaEq3}
 \ v^+(fg_j\mu,\Gamma_1)(z) = \dfrac{(fg_jh)(z)}{2L(z)},~ v^-(fg_j\mu,\Gamma_1)(z) = - \dfrac{(fg_jh)(z)}{2L(z)},~\mathcal \eta-a.a.
 \end{eqnarray}
 for $j \ge 1.$ Apply Theorem \ref{GPTheorem1} and Lemma \ref{PFLipschitz} to $\CT\eta (\lambda)$ for \eqref{NPMLemmaEq2}, we get
 \[
 \begin{aligned}
 \ v^+(w\mathcal H^1,\Gamma_1)(z)v^+(g_j\mu,\Gamma_1)(z) = & v^+(fg_j\mu,\Gamma_1)(z),~ \eta-a.a., \\
 \ v^-(w\mathcal H^1,\Gamma_1)(z)v^-(g_j\mu,\Gamma_1)(z) = & v^-(fg_j\mu,\Gamma_1)(z),~ \eta-a.a.
 \end{aligned}
\]
Combining with \eqref{NPMLemmaEq1} and \eqref{NPMLemmaEq3}, we have
\[
\ g_j(z)h(z) = 0, ~\eta-a.a. \text{ for } j \ge 1,
\]
which implies
\[
\ v^+(g_j\mu,\Gamma_1)(z) = v^-(g_j\mu,\Gamma_1)(z) = 0, ~\eta-a.a. \text{ for } j \ge 1.
\]
This is a contradiction. The lemma is proved.
 \end{proof}

Define
 \[
\ \mathcal E_N(\Lambda) = \left \{\lambda : ~\lim_{\epsilon \rightarrow 0} \mathcal C_\epsilon(g_j\mu)(\lambda )\text{ exists, } \max_{1\le j\le N} |\mathcal C (g_j\mu)(\lambda ) | \le \frac{1}{N} \right \}.
 \]

\begin{theorem}\label{FCharacterization}
There is a subset $\mathcal Q \subset \mathbb C$ with $\gamma(\mathcal Q) = 0$ such that if $\lambda \in\mathbb C \setminus \mathcal Q$, then $\lambda \in \mathcal F(\Lambda)$ if and only if   
 \begin{eqnarray}\label{FCEq4}
 \ \underset{\delta\rightarrow 0}{\overline{\lim}}\dfrac{\gamma(\D(\lambda, \delta)\cap\mathcal E_N(\Lambda))}{\delta} > 0 
 \end{eqnarray}
for all $N \ge 1$. Consequently, $\mathcal F(\Lambda)$ and $\mathcal R(\Lambda)$ do not depend on the choices of $\{\Gamma_n\}$ (as in Lemma \ref{GammaExist}), up to a set of zero analytic capacity. 
\end{theorem}

\begin{proof}
We first prove that there exists $\mathcal Q_1$ with $\gamma(\mathcal Q_1) = 0$ such that if $\lambda_0\in \mathcal {ZD}(\mu) \setminus \mathcal Q_1$, then $\lambda_0\in \mathcal F(\Lambda)$ if and only if $\lambda_0$ satisfies \eqref{FCEq4}.

From Definition \ref{R0ForRILambda}, we get $\mathcal {ZD}(\mu) \approx \mathcal R_0 (\Lambda) \cup \mathcal F_0 (\Lambda), ~\gamma-a.a.$. There exists $\mathcal Q_1$ with $\gamma(\mathcal Q_1) = 0$ such that for $\lambda_0\in \mathcal {ZD}(\mu) \setminus \mathcal Q_1$, $\Theta_{ g_j\mu}(\lambda _0) = 0$,
$\mathcal C(g_j\mu) (\lambda_0) = \lim_{\epsilon\rightarrow 0} \mathcal C_\epsilon (g_j\mu) (\lambda_0)$ exists, and $\mathcal C(g_j\mu) (z)$ is $\gamma$-continuous at $\lambda_0$ (Lemma \ref{CauchyTLemma})  for all $j\ge 1$.

If $\lambda_0\in \mathcal R_0(\Lambda)$, then there exists $j_0$ such that $\mathcal C(g_{j_0}\mu) (\lambda_0) \ne 0$. Let $\epsilon_0 = \frac 12 |\mathcal C(g_{j_0}\mu) (\lambda_0)|$, we obtain that, for $N > N_0 := \max(j_0,\frac{1}{\epsilon_0} + 1)$,
 \[
 \ \mathcal E_N(\Lambda) \subset \{|\mathcal C(g_{j_0}\mu) (z) - \mathcal C(g_{j_0}\mu) (\lambda_0)| > \epsilon_0 \}.
 \]
Therefore, by Lemma \ref{CauchyTLemma}, for $N > N_0$,  
 \[
 \ \lim_{\delta\rightarrow 0} \dfrac{\gamma (\D(\lambda_0 , \delta ) \cap \mathcal E_N(\Lambda))}{\delta} = 0.
 \]
Thus, $\lambda_0$ does not satisfy \eqref{FCEq4}. 

Now for $\lambda_0\in \mathcal F_0(\Lambda)$, $\mathcal C(g_j\mu) (\lambda_0) = 0$ for all $j \ge 1$.
Using Lemma \ref{CauchyTLemma} and  Theorem \ref{TolsaTheorem} (2), we get
 \[
 \ \begin{aligned}
 \ &\lim_{\delta\rightarrow 0} \dfrac{\gamma (\D(\lambda _0, \delta) \setminus \mathcal E_N(\Lambda))}{\delta } \\
 \ \le & A_T \sum_{j=1}^N \lim_{\delta\rightarrow 0} \dfrac{\gamma (\D(\lambda _0, \delta) \cap \{|\mathcal C(g_j\mu)(z) - \mathcal C(g_j\mu) (\lambda_0)| \ge \frac{1}{N}\})}{\delta } \\
 \ = & 0. 
 \ \end{aligned}
 \]
Hence, $\lambda_0$ satisfies \eqref{FCEq4}.

We now prove that there exists $\mathcal Q_2$ with $\gamma(\mathcal Q_2) = 0$ such that if $\lambda_0\in \mathcal {ND}(\mu) \setminus \mathcal Q_2$, then $\lambda_0\in \mathcal F(\Lambda)$ if and only if $\lambda_0$ satisfies \eqref{FCEq4}.

From Definition \ref{FPMR1Def}, we get $\mathcal {ND}(\mu) \approx \mathcal R_1 (\Lambda) \cup (\mathcal F_+(\Lambda) \cup \mathcal F_-(\Lambda)), ~\gamma-a.a.$. There exists $\mathcal Q_2$ with $\gamma(\mathcal Q_2) = 0$ such that for $\lambda_0\in \mathcal {ND}(\mu) \cap \Gamma_n\setminus \mathcal Q_2$, 
$v^0(g_j\mu, \Gamma_n)(\lambda_0) = \mathcal C(g_j\mu) (\lambda_0) = \lim_{\epsilon\rightarrow 0} \mathcal C_\epsilon (g_j\mu) (\lambda_0)$, $v^+(g_j\mu, \Gamma_n, \beta_n)(\lambda_0)$, and $v^-(g_j\mu, \Gamma_n, \beta_n)(\lambda_0)$ exist for all $j,n \ge 1$ and Theorem \ref{GPTheorem1} (b), (c), and (d) hold. Fix $n = 1$ and without loss of generality, we assume $\beta_1 =0.$

If $\lambda_0\in \mathcal R_1(\Lambda)$, then there exist integers $j_0$, $j_1$, and $j_2$ such that $v^0(g_{j_0}\mu, \Gamma_1)(\lambda_0) \ne 0$ by Lemma \ref{NPMLemma} (in fact, we can let $j_0 = j_1$ or $j_0 = j_2$), $v^+(g_{j_1},\Gamma_1)(\lambda_0) \ne 0$, and $v^-(g_{j_2}\mu, \Gamma_1)(\lambda_0) \ne 0$. Set
 \[
 \ \epsilon_0 = \dfrac 12 \min (|v^0(g_{j_0}\mu, \Gamma_1)(\lambda_0)|, |v^+(g_{j_1},\Gamma_1)(\lambda_0)|, |v^-(g_{j_2}\mu, \Gamma_1)(\lambda_0)|), 
 \]
then for $N > N_0 := \max (j_0, j_1, j_2, \frac{1}{\epsilon_0} + 1)$,
\[
\ \Gamma_1\cap \mathcal E_N (\Lambda ) \subset D := \Gamma_1 \cap \{|\mathcal{C}(g_{j_0}\mu) (z ) - v^0(g_{j_0}\mu, \Gamma_1)(\lambda_0)| \ge \epsilon_0 \},
 \] 
\[
\ U_{\Gamma_1}\cap \mathcal E_N (\Lambda ) \subset
 E := U_{\Gamma_1}\cap \{|\mathcal{C}(g_{j_1}\mu)(z ) - v^+(g_{j_1}\mu, \Gamma_1)(\lambda_0)| \ge \epsilon_0 \},
 \]
and
 \[
\ L_{\Gamma_1}\cap \mathcal E_N (\Lambda ) \subset
 F :=  L_{\Gamma_1}\cap \{|\mathcal{C}(g_{j_2}\mu)(z ) - v^-(g_{j_2}\mu, \Gamma_1)(\lambda_0)| \ge \epsilon_0 \}.
 \]
Therefore, using Theorem \ref{TolsaTheorem} (2) and Theorem \ref{GPTheorem1}, we get for $N > N_0$,
 \[
 \begin{aligned}
 \ &\lim_{\delta\rightarrow 0}\dfrac{\gamma(\D(\lambda_0, \delta) \cap \mathcal E_N (\Lambda ))}{\delta} \\
 \ \le & A_T\left ( \lim_{\delta\rightarrow 0}\dfrac{\gamma(\D(\lambda_0, \delta) \cap D)}{\delta} + \lim_{\delta\rightarrow 0}\dfrac{\gamma(\D(\lambda_0, \delta) \cap E)}{\delta} + \lim_{\delta\rightarrow 0}\dfrac{\gamma(\D(\lambda_0, \delta) \cap F)}{\delta}\right ) \\
 \ = &0.
 \end{aligned}
 \]
Hence, $\lambda_0$ does not satisfy \eqref{FCEq4}.   
 
For $\lambda_0\in (\mathcal F_+ (\Lambda ) \cup \mathcal F_-(\Lambda )) \cap \Gamma_1$, we may assume that $\lambda_0\in \mathcal Z_+(\Lambda , \Gamma_1)\subset \Gamma_1$. Using Theorem \ref{TolsaTheorem} (2) and Theorem \ref{GPTheorem1}, we get ($v^+(g_j\mu, \Gamma_1)(\lambda_0) = 0$)
 \[
 \ \begin{aligned}
 \ &\lim_{\delta\rightarrow 0}\dfrac{\gamma(\D(\lambda_0, \delta) \cap U_{\Gamma_1}\setminus \mathcal E_N (\Lambda ))}{\delta} \\
 \ \le & A_T \sum_{j=1}^N\lim_{\delta\rightarrow 0}\dfrac{\gamma(\D(\lambda_0, \delta) \cap U_{\Gamma_1}\cap \{|\mathcal{C}(g_j\mu)(z ) - v^+(g_j\mu, \Gamma_1)( \lambda_0)| \ge \frac{1}{N} \})}{\delta} \\
 \ \ = & 0.
 \ \end{aligned}
 \]
This implies
 \[
 \ \underset{\delta\rightarrow 0}{\overline \lim}\dfrac{\gamma(\D(\lambda_0, \delta) \cap \mathcal E_N(\Lambda )}{\delta} \ge \underset{\delta\rightarrow 0}{\overline \lim}\dfrac{\gamma(\D(\lambda_0, \delta) \cap U_{\Gamma_1}\cap \mathcal E_N(\Lambda ))}{\delta} > 0.
\] 
Hence, $\lambda_0$ satisfies \eqref{FCEq4}.

Finally, \eqref{FCEq4} does not depend on choices of $\{\Gamma_n\}$, therefore, $\mathcal F(\Lambda )$ and $\mathcal R(\Lambda )$ are independent of choices of $\{\Gamma_n\}$ up to a set of zero analytic capacity.  
\end{proof}

\section{\textbf{Building block functions and uniqueness of $\mathcal F(\Lambda)$ and $\mathcal R(\Lambda)$}}

The following is our main lemma, which extends \cite[Lemma 5.1]{cy22} to $\Lambda.$

\begin{lemma} \label{BBFRLambda} 
Let $E_1\subset \mathcal F(\Lambda)$ be a compact subset with $\gamma(E_1) > 0$. Then 
there exists $f\in R^{t,\i}(K, \mu),$ $\eta\in M_0^+(E_1),$ and $\|\mathcal C_\epsilon (\eta) \| \le C_{12}$ such that $f(z) = \mathcal C(\eta)(z)$ for $z\in \C_\i \setminus E_1$, 
 \[
 \ \|f\|_{L^\infty(\mu)} \le C_{12},~ f(\infty) = 0,~ f'(\infty) = - \gamma(E_1),
 \]
and
 \begin{eqnarray}\label{BBFRLambdaEq1}
 \ \mathcal C(\eta)(z) \mathcal C(g_j\mu) (z) = \mathcal C(fg_j\mu) (z), ~\gamma|_{E_1^c}-a.a. \text{ for }j \ge 1.
 \end{eqnarray} 
\end{lemma}

From the semi-additivity of $\gamma$ (see Theorem \ref{TolsaTheorem} (2)), we have
\[
\ \gamma(E_1) \le A_T(\gamma(E_1 \cap \mathcal F_0(\Lambda)) + \gamma(E_1 \cap \mathcal F_+(\Lambda)) + \gamma(E_1 \cap \mathcal F_-(\Lambda))).
\]
Hence,
\[
\ \max(\gamma (E_1 \cap \mathcal F_0 (\Lambda) ),~ \gamma (E_1 \cap \mathcal F_+ (\Lambda) ),~ \gamma (E_1 \cap \mathcal F_- (\Lambda) )) \ge \dfrac{1}{3A_T}\gamma (E_1).
 \]
 Therefore, we shall prove Lemma \ref{BBFRLambda} assuming $E_1 \subset \mathcal F_0 (\Lambda),$ $E_1 \subset \mathcal F_+ (\Lambda),$ or $E_1 \subset \mathcal F_- (\Lambda).$ 
In the following proofs, we fix the measure $\eta = \eta_1$ and the function $f = f_1$ from Lemma \ref{BBFunctLemma} for $\{g_n\} = \Lambda$ and $E=E_1$.

\begin{proof} (Lemma \ref{BBFRLambda} assuming $E_1 \subset \mathcal F_0 (\Lambda)$):  From \eqref{BBFunctLemmaEq1} and \eqref{BBFunctLemmaEq2}, we see that $f_1\in R^{t,\i}(K, \mu)$ and $f_1(\lambda)\mathcal C(g_j\mu) (\lambda) = \mathcal C(f_1g_j\mu) (\lambda,~\gamma |_{E_1^c}-a.a.$ for $j \ge 1.$
 Set 
 \[
 \ f = \dfrac{f_1}{\|\eta_1\|} \gamma (E_1) \text{ and } \eta  = \dfrac{\eta_1}{\|\eta_1\|} \gamma (E_1). 
 \]
 Then $f$ and $\eta$ satisfy the properties of the lemma.
 \end{proof}
 
 \begin{proof} (Lemma \ref{BBFRLambda} assuming $E_1 \subset \mathcal F_+ (\Lambda)$): Using Corollary \ref{ACIncreasing}, we assume that there exists a positive integer $n_0$ such that 
$E_1\subset \bigcup_{n=1}^{n_0}\Gamma_n.$
 Put $E_1 = \cup_{n=1}^{n_0}F_n,$ where $F_n \subset \Gamma_n$ and $F_n \cap F_m = \emptyset$ for $n \ne m.$ Since $\eta_1$ is of linear growth, we can set $\eta_1 = \sum_{n=1}^{n_0} w_n(z)\mathcal H^1 |_{\Gamma_n},$ where $w_n$ is supported on $F_n.$ The function 
 \[
 \ f_2(z) = f_1(z) - \frac{1}{2} \sum_{n=1}^{n_0} e^{-i\beta_n}L((z - z_n)e^{-i\beta_n} + z_n)^{-1}w_n(z).
 \] 
 is the non-tangential limit of $f_1$ from the bottom on $E_1$ (Lemma \ref{PFLipschitz}). Thus, 
 $f_2,w_n\in L^\infty(\mu)$ and $\|f_2\|_{L^\infty(\mu)} \le C_{12}.$
 From \eqref{BBFunctLemmaEq1}, we get
 \begin{eqnarray}\label{BBFGEq3} 
 \begin{aligned}
 \ \int f_2 g_jd\mu = &  - \int \mathcal C(g_j\mu)  d\eta_1 - \frac{1}{2} \sum_{n=1}^{n_0} \int_{F_n} e^{-i\beta_n}L((z - z_n)e^{-i\beta_n} + z_n)^{-1}g_j h d\eta_1 \\
 \    = & - \sum_{n=1}^{n_0} \int_{F_n} v^+(g_j\mu, \Gamma_n, \beta_n) d\eta_1.
 \end{aligned}
 \end{eqnarray}
 Similarly, for $\lambda\in E_1^c$ and by \eqref{BBFunctLemmaEq2}, we have
 \begin{eqnarray}\label{BBFGEq4} 
 \ \int \dfrac{f_2(z)-f_2(\lambda)}{z-\lambda} g_j(z)d\mu (z) = - \sum_{n=1}^{n_0} \int_{F_n} v^+(g_j\mu, \Gamma_n, \beta_n)(z)\dfrac{d\eta_1(z)}{z-\lambda} .
 \end{eqnarray}
 Since $v^+(g_j\mu, \Gamma_n, \beta_n)(z) = 0, ~z \in F_n,~ \mathcal H^1 |_{\Gamma_n}-a.a.,$ by \eqref{BBFGEq3}, we get 
 $f_2\in R^{t,\i}(K, \mu).$
 Similarly, from \eqref{BBFGEq4}, we see that $\frac{f_2(z)-f_2(\lambda)}{z-\lambda}\in R^{t,\i}(K, \mu)$ and  $f_2(\lambda)\mathcal C(g_j\mu) (\lambda) = \mathcal C((f_2 g_j\mu) (\lambda)$ for $\lambda\in E_1^c,~\gamma-a.a.$ 
  Set 
 \[
 \ f = \dfrac{f_2}{\|\eta_1\|} \gamma (E_1) \text{ and } \eta  = \dfrac{\eta_1}{\|\eta_1\|} \gamma (E_1). 
 \]
 Then $f$ and $\eta$ satisfy the properties of the lemma.
 \end{proof}
 
 The proof of Lemma \ref{BBFRLambda} for $E_1 \subset \mathcal F_- (\Lambda)$ is the same as that for $E_1 \subset \mathcal F_+ (\Lambda)$ if we modify the definition of $f_2$ by the following
\[
 \ f_2(z) = f_1(z) + \frac{1}{2} \sum_{n=1}^{n_0} e^{-i\beta_n}L((z - z_n)e^{-i\beta_n} + z_n)^{-1}w_n(z).
 \] 
 
 We now show the properties of $g_j\mu$ on $\mathcal F(\Lambda)$ are preserved for all annihilating measures $g\mu$ for $g\in R^t(K, \mu)^\perp.$  

\begin{theorem}\label{FLambdaForW}
Let $\Gamma_n$, $\Gamma$, and $d\mu = hd\mathcal H^1 |_{\Gamma} + d\mu_s$ be as in Lemma \ref{GammaExist}. Then
for $g\in R^t(K, \mu)^\perp,$
\begin{eqnarray}\label{FLambdaForWEq1}
\ \mathcal C(g\mu)(\lambda) = 0,~ \gamma|_{\mathcal F_0(\Lambda)}-a.a.,  
\end{eqnarray}
\begin{eqnarray}\label{FLambdaForWEq2}
\ v^+(g\mu, \Gamma_n, \beta_n)(\lambda) = 0, ~ \gamma |_{\mathcal F_+(\Lambda)\cap \Gamma_n}-a.a. ~\text{for}~  n \ge 1,  
\end{eqnarray}
and
\begin{eqnarray}\label{FLambdaForWEq3}
\ v^-(g\mu, \Gamma_n,  \beta_n)(\lambda) = 0,~ \gamma |_{\mathcal F_-(\Lambda)\cap \Gamma_n}-a.a. ~\text{for}~  n \ge 1.  
\end{eqnarray}
\end{theorem}

\begin{proof} Equation \eqref{FLambdaForWEq1}: Assume that there exists a compact subset $E_1\subset \mathcal F_0 (\Lambda)$ such that $\gamma (E_1) > 0$ and  
 \begin{eqnarray}\label{FCForRAssump}
 \ Re(\mathcal C(g\mu)) > 0,~ \lambda\in E_1.
 \end{eqnarray}
  Using Lemma \ref{CTMaxFunctFinite}, we assume that $\mathcal C_*(g\mu) (\lambda) < M_1$ for $\lambda\in E_1$. Let $f\in R^{t,\i}(K, \mu)$ and $\eta \in M_0^+(E_1)$ be as in Lemma \ref{BBFRLambda}. 
  By \eqref{BBFunctLemmaEq1}, we get
  \[
  \ \int Re(\mathcal C(g\mu)(z)) d\eta(z) = - Re \left ( \int f(z) g(z) d \mu(z)\right ) = 0,
  \]
  which implies $Re(\mathcal C(g\mu)(z)) = 0,~ \eta-a.a..$ This contradicts to \eqref{FCForRAssump}.
  
   Equation \eqref{FLambdaForWEq2}: Assume that there exists a compact subset $E_1\subset \mathcal F_+ (\Lambda) \cap \Gamma_1$ such that $\gamma (E_1) > 0$ and  
 \begin{eqnarray}\label{FCForRAssump2}
 \ Re(v^+(g\mu, \Gamma_1, \beta_1)(z)) > 0,~ \lambda\in E_1.
 \end{eqnarray}
  Using Lemma \ref{CTMaxFunctFinite}, we assume that $\mathcal C_*(g\mu) (\lambda) < M_1$ for $\lambda\in E_1$. Let $f\in R^{t,\i}(K, \mu)$ and $\eta \in M_0^+(E_1)$ be as in Lemma \ref{BBFRLambda}. Similar to \eqref{BBFGEq3}, we see that
   \[
  \ \int Re(v^+(g\mu, \Gamma_1, \beta_1)(z)) d\eta(z) = - Re \left ( \int f(z) g(z) d \mu(z)\right ) = 0,
  \]
  which implies $Re(v^+(g\mu, \Gamma_1, \beta_1)(z)) = 0,~ \eta-a.a..$ This contradicts to \eqref{FCForRAssump2}.
  
  The proof of the equation \eqref{FLambdaForWEq3} is the same as above.
\end{proof}

As a simple application of Theorem \ref{FLambdaForW} \eqref{FLambdaForWEq1} and the fact that $\area (\mathcal F_+(\Lambda)\cup \mathcal F_-(\Lambda)) = 0$, we have the corollary below.

\begin{corollary} \label{acZero} 
For $g\perp R^t(K, \mu)$,
$\mathcal C(g\mu)(z) = 0, ~ \area |_{\mathcal F(\Lambda)}-a.a.$.
\end{corollary}

The following corollary is a direct outcome of Theorem \ref{FLambdaForW}.

\begin{corollary}\label{FUnique} 
Let $\Lambda ' = \{g_n'\}_{n=1}^\infty \subset R^t(K, \mu)^\perp$ be a dense subset. 
Then
 \[
 \ \mathcal F_0(\Lambda) \approx \mathcal F_0 (\Lambda '), ~ \mathcal F_+ (\Lambda)\approx \mathcal F_+(\Lambda '), ~ \mathcal F_- (\Lambda) \approx \mathcal F_-(\Lambda '), ~ \gamma-a.a..
 \]
Hence, $\mathcal F(\Lambda)$ and $\mathcal R(\Lambda)$ are independent of choices of $\Lambda$ up to a $\gamma$ zero set.
\end{corollary}

\section{\textbf{Analytic bounded point evaluations for $\rtkmu$}}

Let $\Gamma_n$, $\Gamma$, and $d\mu = hd\mathcal H^1 |_{\Gamma} + d\mu_s$ be as in Lemma \ref{GammaExist}. From Theorem \ref{FCharacterization} and Corollary \ref{FUnique}, $\mathcal F(\Lambda)$ and $\mathcal R(\Lambda)$ are independent of choices of $\Gamma_n$ and $\Lambda$ up to a zero set of analytic capacity. This leads the following definition for $\rtkmu.$

\begin{definition}\label{FRForRtDef}
We define the non-removable boundaries and the removable sets for $\rtkmu$ as the following:
\[
\begin{aligned}
\ & \mathcal F_0 = \mathcal F_0(\Lambda),~ \mathcal F_+ = \mathcal F_+(\Lambda),~ \mathcal F_- = \mathcal F_-(\Lambda),\text{ and }\mathcal F = \mathcal F(\Lambda). \\
\ & \mathcal R_0 = \mathcal R_0(\Lambda),~ \mathcal R_1 = \mathcal R_1(\Lambda),\text{ and }\mathcal R = \mathcal R(\Lambda).
\end{aligned}
\]
Set $\mathcal E_N = \mathcal E_N(\Lambda).$	
\end{definition}

\begin{proposition} \label{RTFRProp1} 
 If $S_\mu$ on $\rtkmu$ is pure, then 
$\text{spt}(\mu) \subset \overline{\mathcal R}.$
\end{proposition}

\begin{proof}
If $\D(\lambda_0,\delta) \subset \mathcal F$, from Corollary \ref{acZero}, we see that $\mathcal C(g_j\mu)(z) = 0$ with respect to $\area|_{\D(\lambda_0,\delta)}-a.a.$ This implies that $\mu(\D(\lambda_0,\delta)) = 0$ by \eqref{CTDistributionEq} since $S_\mu$ is pure.
\end{proof}

\begin{proposition} \label{RTFRProp2} 
$\partial_1 K \subset \mathcal F, ~\gamma-a.a.$.
\end{proposition}

\begin{proof}
The proposition follows from Theorem \ref{FCharacterization} and the fact that for $\lambda \in \partial_1 K$, 
 \[
 \ \D(\lambda, \delta) \setminus K \subset \D(\lambda, \delta) \cap\mathcal E_N.
 \]
 \end{proof}
 
 The following Lemma is from Lemma B in \cite{ars02}.

\begin{lemma} \label{lemmaARS}
There are absolute constants $\epsilon _1, C_{13} > 0$ with the
following property. If $R > 0$ and $E \subset  \overline{\D(0, R)}$ with 
$\gamma(E) < R\epsilon_1$, then for $|\lambda| < \frac{R}{2}$ and polynomials $p$,
\[
\ |p(\lambda)| \le \dfrac{C_{13}}{\pi R^2} \int _{\overline{\D(0, R)}\setminus E} |p|\, d\area.
\]
\end{lemma}

The theorem below provides  an important connection between $\text{abpe}(R^t(K,\mu))$ and $\mathcal R.$ 

\begin{theorem}\label{ABPETheoremRT} 
The following property holds:
 \begin{eqnarray}\label{ABPETheoremRTEq1}
 \ \text{abpe}(R^t(K,\mu)) \approx \text{int}(K) \cap \mathcal R,~ \gamma-a.a..
 \end{eqnarray}
More precisely, the following statements are true:

(1) If $\lambda_0 \in \text{int}(K)$ and there exists $N\ge 1$ such that 
 \begin{eqnarray}\label{ABPETheoremRTEq2}
 \ \lim_{\delta \rightarrow 0} \dfrac{ \gamma (\mathcal E_N \cap \D(\lambda_0, \delta))}{\delta} = 0,
 \end{eqnarray}
then $\lambda_0\in \text{abpe}(R^t(K,\mu))$.

(2)
\[
 \  \text{abpe}(R^t(K,\mu)) \subset \text{int}(K) \cap \mathcal R,~ \gamma-a.a..
 \]
\end{theorem}

\begin{proof}
(1): If $\lambda_0 \in \text{int}(K)$  satisfies \eqref{ABPETheoremRTEq2}, then we  choose $\delta > 0$ small enough such that $\D(\lambda_0, \delta) \subset int(K)$ and 
$\gamma (E:= \mathcal E_N \cap \D(\lambda_0, \delta)))\le \epsilon_1 \delta$, where $\epsilon_1$ is from Lemma \ref{lemmaARS}. Hence, using Lemma \ref{lemmaARS}, we conclude
 \[
 \ \begin{aligned}
 \ |r(\lambda )| \le & \dfrac{C_{13}}{\pi \delta^2} \int _{\D(\lambda_0, \delta) \setminus E} |r(z)| d\area(z) \\
 \ \le & \dfrac{NC_{13}}{\pi \delta^2} \int _{\D(\lambda_0, \delta)} |r(z)| \max_{1\le j \le N} |\mathcal C(g_j\mu)(z)| d\area(z) \\
\ \le & \dfrac{NC_{13}}{\pi \delta^2} \sum_{j = 1}^N\int _{\D(\lambda_0, \delta)}  |\mathcal C(rg_j\mu)(z)| d\area(z) \\
\ \le & \dfrac{NC_{13}}{\pi \delta^2} \sum_{j = 1}^N\int \int _{\D(\lambda_0, \delta)} \left |\dfrac{1}{z-w} \right | d\area(z) |r(w)||g_j (w)| d\mu (w) \\
\ \le & \dfrac{NC_{14}}{\delta} \sum_{j = 1}^N \|g_j\|_{L^{s}(\mu)} \|r\|_{L^{t}(\mu)}
 \ \end{aligned}
 \]
for all $\lambda$ in $\D(\lambda_0, \frac{\delta}{2})$ and all $r\in Rat(K)$. This implies that $\lambda_0\in abpe(R^t(K,\mu))$.

(2): Let $E \subset G\cap \mathcal F$ be a compact subset with $\gamma(E) > 0$, where $G$ is a connected component of $abpe(R^t(K,\mu))$. By Lemma \ref{BBFRLambda}, there exists $f\in R^{t,\i}(K,\mu)$ that is bounded analytic on $E^c$ such that  
 $\|f\|_{L^\infty(\mu)} \le C_{12},~ f(\infty) = 0,~ f'(\infty) = \gamma(E),$
and 
$\dfrac{f(z) - f(\lambda)}{z - \lambda} \in R^{t,\i}(K,\mu)$ for $\lambda \in E^c.$
Let $r_n\in Rat(K)$ such that $\|r_n-f\|_{L^t(\mu)} \rightarrow 0$. Hence, $r_n$ uniformly tends to an analytic function $f_0$ on compact subsets of $G$ and $\frac{f-f_0(\lambda)}{z-\lambda} \in R^t(K,\mu)$ for $\lambda\in G$. 
For $\lambda\in G\setminus E$ with $\mu(\{\lambda\})=0,$ let $k_\lambda \in L^s(\mu)$ such that $r(\lambda) = (r,k_\lambda)$ for $r\in \text{Rat}(K).$ Hence, $f_0(\lambda)= (f,k_\lambda).$ Set $g= (z-\lambda)\bar k_\lambda.$ Then $g \perp \rtkmu$ and
\[
\ 0 = \int \dfrac{f(z) - f(\lambda)}{z - \lambda} gd\mu = (f - f(\lambda),k_\lambda) = (f,k_\lambda) - f(\lambda).
\]
Therefore, $f(z) = f_0(z)$ for $z\in G\setminus E.$  Thus, the function $f_0(z)$ can be analytically extended to $\mathbb C_\i$ and $f_0(\infty) = 0$. So $f_0=0.$ This contradicts to $f'(\infty) \ne 0$. 

\eqref{ABPETheoremRTEq1} now follows from Theorem \ref{FCharacterization}.
\end{proof}

\section{\textbf{Proof of Theorem A and Theorem B}}

For  a compact subset $E \subset \D(a, \delta)$ and an analytic function $f$ on $\C_\i \setminus E$ with $f(\infty) = 0,$ we consider the Laurent expansion of $f$ centered at $a$ for $z\in \mathbb C\setminus \D(a, \delta),$

\[
 \ f(z) = \sum_{n=1}^\infty \dfrac{c_n(f,a)}{(z-a)^n}.
 \]
 As $c_1(f, a)$ does not depend on the choice of $a$,  we define 
 $c_1 (f) = c_1(f, a).$
 
 Let $\varphi$ be a smooth function with compact support. Vitushkin's localization operator
$T_\varphi$ is defined by
 \[
 \ (T_\varphi f)(\lambda) = \dfrac{1}{\pi}\int \dfrac{f(z) - f(\lambda)}{z - \lambda} \bar\partial \varphi (z) d\area(z),
 \]
where $f\in L^1_{loc} (\C)$.	 Clearly, $(T_\varphi f)(z) = - \frac{1}{\pi}\mathcal C(\varphi \bar\partial f\area ) (z).$
Therefore, by \eqref{CTDistributionEq}, $T_\varphi f$ is analytic outside of $\text{supp} (\bar \partial f) \cap\text{supp} (\varphi).$ If $\text{supp} (\varphi) \subset \D(a,\delta),$ then  
\[
 \ \| T_\varphi f\|_\i   \le  4\|f\|_\i  \delta\|\bar\partial \varphi\|.
 \]	
 See \cite[VIII.7.1]{gamelin} for the details of $T_\varphi.$

Let $\delta > 0$. We say that 
$\{\varphi_{ij},S_{ij}, \delta\}$ is a smooth partition of unity subordinated to $\{2S_{ij}\},$ 
 if $S_{ij}$ is a square with vertices   $(i\delta,j\delta),~((i+1)\delta,j\delta),~(i\delta,(j+1)\delta),$ and $((i+1)\delta,(j+1)\delta);$
$\varphi_{ij}$ is a $C^\infty$ smooth function supported in $2S_{ij},$
and with values in $[0,1]$; and
 \[
 \ \|\bar\partial \varphi_{ij} \| \le \frac{C_{15}}{\delta},~ \sum \varphi_{ij} = 1.
 \]
We shall let $c_{ij}$ denote the center of $S_{ij}$ (See \cite[VIII.7]{gamelin} for details).

Let $f\in C(\C_\i)$ (or $f\in L^\i (\C)$) with compact support.  Define $f_{ij} = T_{\varphi_{ij}}f,$ then $f_{ij} \ne 0$ for only finite many $(i,j).$ Clearly,
$f(z) = \sum_{f_{ij}\ne 0}f_{ij}(z).$
The standard Vitushkin approximation scheme requires us to construct functions $a_{ij}$ such that $f_{ij} - a_{ij}$ has triple zeros at $\infty$, which requires us to estimate both $c_1 (a_{ij})$ and $c_2(a_{ij}, c_{ij})$ (see \cite[section 7 on page 209]{gamelin}). 
The main idea of P. V. Paramonov \cite{p95} is that one does not actually need to estimate
each coefficient $c_2(a_{ij}, c_{ij})$. It suffices to estimate 
the sum of coefficients $\sum_{j\in J_{il}} c_2(a_{ij}, c_{ij})$ for a special non-intersecting partition $\{J_{il}\}.$ Let 
\begin{eqnarray}\label{ThetaIJ}
\ \alpha_{ij} = \gamma (\D(c_{ij}, k_1\delta) \cap E),
\end{eqnarray} 
where $k_1\ge 3$ is a fixed integer.
Let $m_1 =  \min\{i,j,~  \D(c_{ij}, k_1\delta) \cap E \ne \emptyset\}$ and
$m_2 = \max\{i,j,~  \D(c_{ij}, k_1\delta) \cap E \ne \emptyset\}.$
So $m_1$ and $m_2$ are finite.
Set $min_i = \min\{j:~  \alpha_{ij}\ne 0 \}$ and $max_i = \max\{j:~ \alpha_{ij}\ne 0 \}$. Let $J_i = \{j:~ min_i \le j \le max_i\}$. 

\begin{definition}
We call a subset $J$ of $J_i$ a  complete group if 
$J = \{j:~ j_1 + 1\le j \le j_1 + s_1 + s_2+s_3 \},$
where $s_2$ is an absolute constant that will be chosen in Lemma \ref{BasicEstimate3},
and $s_1$ and $s_3$ are chosen so that
 \[
 \ \delta \le \sum_{j=j_1+1}^{j_1+s_1}\alpha_{ij} < \delta + k_1\delta
\text{ and } \delta \le \sum_{j=j_1+s_1+s_2+1}^{j_1+s_1+s_2+s_3}\alpha_{ij} < \delta + k_1\delta.
 \]
 \end{definition}

 We now present a detailed description of the procedure of partitioning $J_i$ into
groups. We split each $J_i$ into (finitely many) non-intersecting groups $J_{il}$, $l = 1,...,l_i$, as follows.
Starting from the lowest index $min_i$ in $J_i$ we include in $J_{i1}$ (going upwards and without jumps in $J_i$) all indices until we have collected a minimal (with respect to the number of elements) complete group $J_{i1}$. Then we repeat this
procedure for $J_i\setminus J_{i1}$, and so on. After we have constructed all the complete
groups $J_{i1},...,J_{il_i-1}$ in this way  (there may be none), then  what remains is the last
portion $J_{il_i} = J_i\setminus(J_{i1}\cup...\cup J_{il_i-1})$ of indices in $J_i$, which includes no complete
groups. We call this portion $J_{il_i}$ an incomplete group of indices (clearly, there is at most one incomplete group for each $i$).

\begin{definition}\label{MVSPDef}
Let $\delta > 0,$ $\alpha_{ij},$ and $E$ be defined as in \eqref{ThetaIJ}. 
Let $P_1$ be the set of all complete groups $J_{il}$ for $1\le l \le l_i -1$ and  $m_1 \le i \le m_2.$ Let $P_2$ be the set of all incomplete groups $J_{il_i}$ for $m_1 \le i \le m_2.$
 Set $P = P_1 \cup P_2.$
\end{definition}

For a group $J$ (complete or incomplete) with row index $i,$ 
let 
\begin{eqnarray}\label{LIDef0}
\ J'(z) = \{j\in J:~\alpha_{ij} > 0 \text{ and } |z-c_{ij}| >3k_1\delta  \},
\end{eqnarray}
and
 \begin{eqnarray}\label{LIDef}
 \ L_J' (z) = \sum_{j\in J'(z)} \left ( \dfrac{\delta\alpha_{ij}}{|z - c_{ij}|^2} + \dfrac{\delta^3}{|z - c_{ij}|^3} \right ).
 \end{eqnarray}
Define $L_J(z) = L_J' (z)$ if $\{j\in J:~\alpha_{ij} > 0\} = J'(z)$, otherwise, $L_J(z) = 1+L_J' (z)$. For $g_{ij}$ that is bounded and analytic on $\mathbb C\setminus E_{ij},$ where $E_{ij}$ is a compact subset of $\D(c_{ij}, k_1\delta),$ define
 \begin{eqnarray}\label{GIDefinition}
 \ g_J = \sum_{j\in J} g_{ij}, ~ c_1(g_J) = \sum_{j\in J} c_1(g_{ij}),~c_2 (g_J) = \sum_{j\in J} c_2 (g_{ij},c_{ij}).
 \end{eqnarray}

\begin{definition}\label{GApplicable}
A sequence of functions 	$\{g_{ij}\}$ is applicable to $P$ and $\mu \in M_0^+(\C)$ if the following assumptions hold, 
for some absolute constant $C_{16}$:

(1) $g_{ij}$ is bounded and analytic on $\mathbb C_\i \setminus \overline{\D(c_{ij}, k_1\delta)};$ 

(2) $g_{ij}(\infty) = c_1(g_{ij}) = 0;$ 

(3) $\|g_{ij}\|_{L^\infty(\mu)} \le C_{16}$;

(4) $|c_n(g_{ij}, c_{ij})| \le C_{16} \delta^{n-1} \alpha_{ij}$ for $n \ge 1.$
\end{definition}

The following lemma is straightforward. 

\begin{lemma} \label{LEProp}
If $f(z)$ is bounded analytic on $\mathbb C_\i \setminus \overline{\D(a, \delta)}$ with $f(\i) = 0,$ $\|f\| \le 1,$ $E \subset \overline{\D(a, \delta)},$ and
\[
 \ |c_n(f,a)| \le C_{17} n \delta ^{n-1}\gamma(E).
 \]
 Then for  $|z - a| > 2\delta,$ 
 \[
 \ \left |f(z) - \dfrac{c_1(f)}{z-a}  - \dfrac{c_2(f,a)}{(z-a)^2} \right | \le \dfrac{C_{18}\gamma(E)\delta^2}{|z-a|^3}.
 \]
 \end{lemma}
 Therefore, by Lemma \ref{LEProp}, if $\{g_{ij}\}$ is applicable to $P,$  then for $|z-c_{ij}| >3k_1\delta,$
 \begin{eqnarray}\label{gijEst}
 \ |g_{ij}(z)| \le  \dfrac{C_{19}\delta \alpha_{ij} }{|z-c_{ij}|^2} + \dfrac{C_{19} \delta^3 }{|z-c_{ij}|^3}.
 \end{eqnarray}

The following lemma follows easily from  Definitions \ref{MVSPDef} and \eqref{gijEst}.

\begin{lemma}\label{BasicEstimate}
Let $\{g_{ij}\}$ be applicable to $P$ and $\mu \in M_0^+(\C)$. If $J$ is a group (complete or incomplete) with row index $i,$
 then 
 \[
 \begin{aligned}
 \ & |g_J(z)| \le C_{20} L_J(z), ~ \|g_J\|(\|g_J\|_{L^\infty(\mu)}) \le C_{20},     \\ 
 \ & c_1(g_J) = 0,~|c_2(g_J)| \le C_{20} \delta^2.
 \end{aligned}
 \]
 \end{lemma}
 
 The following key lemma  is due to \cite[Lemma 2.7]{p95}, where $s_2$ is also selected.
	
	\begin{lemma}\label{BasicEstimate3}
Let $\{g_{ij}\}$ be applicable to $P$ and $\mu \in M_0^+(\C).$  Suppose
that there exists $h_{ij} \in L^\infty (\mu)$ that is bounded and
 analytic on $\mathbb C_\i\setminus \overline{\D(c_{ij},k_1\delta)}$ 
 and satisfy
 \[
 \begin{aligned}
 \ & h_{ij}(\infty) = 0,~\|h_{ij}\| \le C_{21}, ~ c_1(h_{ij}) = \alpha_{ij}, \\
 \ & |c_n(h_{ij},c_{ij})| \le C_{21} \delta ^{n-1}\alpha_{ij} \text{ for } n \ge 1.
 \end{aligned}
 \]
Then for each complete group $J_{il}\in P_1$, there exists a function
$h_{J_{il}}$ that is a linear combination of $h_{ij}$ for $j\in J_{il}^d \cup J_{il}^u,$  where $J_{il}^d = (j_1+1,...,j_1+s_1)$ and $J_{il}^u = (j_1+s_1+s_2+1,...,j_1+s_1+s_2+s_3)$, such that for all $z \in \mathbb C,$
 \begin{eqnarray}\label{Eq2.25}
 \begin{aligned}
 \ &|h_{J_{il}}(z)| \le C_{22}  L_{J_{il}}(z),  ~ \|h_{J_{il}}\| \le C_{22}, \\
 \  &c_1 (h_{J_{il}}) = 0, ~ c_2 (h_{J_{il}}) = c_2 (g_{J_{il}}).
 \end{aligned}
 \end{eqnarray}
\end{lemma}

The following lemma is also from \cite{p95} with slight modifications.

\begin{lemma}\label{BasicEstimate2}
Let $\{g_{ij}\}$ be applicable to $P$ and $\mu \in M_0^+(\C).$  Let $\{h_{ij}\}$ satisfy the assumption in Lemma \ref{BasicEstimate3}. 
Suppose that for each $J_{il}\in P_1,$ $h_{J_{il}}$  is constructed as in Lemma \ref{BasicEstimate3}.
 Set $\Psi_{J_{il}}(z) = g_{J_{il}}(z) - h_{J_{il}}(z).$
 Then 
 \begin{eqnarray}\label{FBounded3}
\ \sum_{J_{il}\in P_1} |\Psi_{J_{il}}(z)| + \sum_{J_{il_i}\in P_2} |g_{J_{il_i}}(z)| \le C_{23}\min \left (1, \dfrac{\delta}{dist(z, E)} \right )^\frac15. 
\end{eqnarray}
 \end{lemma}
 
 \begin{proof} We just need to modify the proof in \cite{p95} slightly. We fix $z\in \C.$ Let $N$ be the largest integer that is smaller that $\frac{dist(z, E)}{2\delta}.$ If $N^\frac45 \le 3k_1 + 4s_2,$ then \eqref{FBounded3} is bounded by an absolute constant $C_{23}$ as showed in \cite{p95}.
 
 Assume that $N^\frac45 >  3k_1 + 4s_2.$ We make the following modifications
 \newline
 (1) For $J\in P$ with row index $i,$ we add $\alpha _{ij} > 0$ to the definitions of $J'(z)$ in \eqref{LIDef0} and $L_J(z)$ in the line after in \eqref{LIDef} ($I'(z)$ in \cite{p95}). By \eqref{LIDef}, $L_J(z) = L_J'(z).$ Let $j_J$ be the largest positive integer that is smaller that $\min_{j\in J'(z)} |j - \frac{Im(z)}{\delta}|.$ Then $j_J > N$ and $L_J(z) \le \frac{C_{25}}{(i - \frac{Re(z)}{\delta})^2 + j_J^2}.$
 \newline
 (2) For $P_2,$  we have 
 \[
 \ \sum_{J_{il_i}\in P_2} |g_{J_{il_i}}(z)| \le C_{26} \sum_i \dfrac{1}{(i - \frac{Re(z)}{\delta})^2 + j_{J_{il_i}}^2} \le \dfrac{C_{27}}{N}.
 \]
 \newline
 (3) For $P_1$ with $|i - \frac{Re(z)}{\delta}| \le N^\frac45,$ we have 
 \[
 \ \sum_{\underset{J_{il}\in P_1}{|i - \frac{Re(z)}{\delta}| \le N^\frac45}} |\Psi_{J_{il}}(z)| \le C_{28} \sum_{|i - \frac{Re(z)}{\delta}| \le N^\frac45} \sum_{j >N}\dfrac{1}{(i - \frac{Re(z)}{\delta})^2 + j^2} \le \dfrac{C_{29}}{N^\frac15}.
 \]
 \newline
 (4) For $P_1$ with $|i - \frac{Re(z)}{\delta}| > N^\frac45,$ we modify the case $|m| > 4k_2$ in \cite{p95} (replace $ 4k_2$ by $N^\frac45$) and get 
 \[
 \ \sum_{\underset{J_{il}\in P_1}{|i - \frac{Re(z)}{\delta}| > N^\frac45}} |\Psi_{J_{il}}(z)| \le C_{30} \sum_{m \ge N^\frac45} \dfrac{1}{m^{\frac54}} \le \dfrac{C_{31}}{N^\frac15}.
 \]
 The lemma is proved.
 \end{proof}

We now start our proofs for Theorem A and Theorem B.

\begin{lemma}\label{ImRhoOnto2}
If $U$ is an open subset of $\text{abpe}(\rtkmu)$ satisfying $\partial U \subset \mathcal F,~\gamma-a.a.,$ then, for $f\in H^\infty (U)$ with $\|f\|_U \le 1$ and $f(z) = 0$ for $z\in \C\setminus U,$ there exists $\tilde f\in R^{t,\i}(K,\mu)$ such that $\tilde f(z) = 0,~\mu |_{\overline U^c}-a.a.,$ $f(z)\mathcal C(g\mu)(z) = \mathcal C(\tilde fg\mu)(z),~ \area-a.a.$ for $g\perp R^t(K, \mu),$ and $\|\tilde f\|_{L^\infty(\mu)} \le C_{32}.$
\end{lemma}

\begin{proof}
Let $k_1 = 3$ and $E=\partial U.$ Let $\alpha_{ij}$ be as in \eqref{ThetaIJ}. Let $P,$ $P_1,$ and $P_2$ be as in Definition \ref{MVSPDef}.
We fix $z\in \C\setminus \partial U.$
Then 
\begin{eqnarray}\label{ImRhoOntoEq1}
 \ f = \sum_{2S_{ij}\cap \partial \Omega \ne \emptyset} (f_{ij}:= T_{\varphi_{ij}}f). 
 \end{eqnarray}
 The function $f_{ij}$ is analytic on $\C_\i \setminus (\overline{\D(c_{ij}, \sqrt {2}  \delta)}\cap \partial U)$ and therefore, 
 \begin{eqnarray}\label{ImRhoOnto2Eq1} 
\begin{aligned}
\ |c_n (f_{ij}, c_{ij})| = &\dfrac{1}{\pi}\left | \int f(z) (z - c_{ij})^{n-1}\bar \partial \varphi_{ij} (z) d\area (z)\right | \\
\ \le & \|T_{\varphi_{ij}} ((z - c_{ij})^{n-1}f)\| \gamma (\overline{\D(c_{ij}, \sqrt {2}  \delta)} \cap \partial U). \\
\ \le & C_{33} \delta ^{n-1}\alpha_{ij}.
\end{aligned}
\end{eqnarray}

 Let $g_{ij}^0$ be as in Lemma \ref{BBFRLambda} for a compact subset $E_1 \subset \D(c_{ij}, k_1 \delta) \cap \partial U \subset \mathcal F$ with $\gamma(E_1) \ge \frac 12 \alpha_{ij}$ when $\alpha_{ij} > 0.$
Let $g_{ij}^1 = - \frac{c_1 (f_{ij})}{\gamma(E_1)}g_{ij}^0,$ $g_{ij} = f_{ij} - g_{ij}^1,$ and $h_{ij} = - \frac {\alpha_{ij}}{\gamma(E_1)}g_{ij}^0.$ Clearly, by \eqref{ImRhoOnto2Eq1}, we see that $\{g_{ij}\}$ is applicable to $P$ (see Definition \ref{GApplicable}) and $\{h_{ij}\}$ satisfies the assumptions in Lemma \ref{BasicEstimate3}. The function $g_{J_{il}}$ defined in \eqref{GIDefinition} satisfies  the properties in Lemma \ref{BasicEstimate}. Let $h_{J_{il}}$ be as in Lemma \ref{BasicEstimate3}.

We rewrite \eqref{ImRhoOntoEq1} as the following
 \[
 \ f = \sum_{J_{il}\in P_1} (g_{J_{il}} - h_{J_{il}}) + \sum_{J_{il_i} \in P_2} g_{J_{il_i}} + f_{\delta}
 \]
where
 \begin{eqnarray}\label{FDeltaSOB}
 \ f_{\delta} = \sum_{J_{il}\in P} \sum_{j\in J_{il}} g_{ij}^0 + \sum_{J_{il}\in P_1} h_{J_{il}} \in R^{t,\i}(K,\mu).
 \end{eqnarray}
 From \eqref{FBounded3}, we get
 \[
 \ |f(z) - f_{\delta}(z)| \le C_{23}\min \left (1, \dfrac{\delta}{dist(z, \partial U)} \right )^\frac15.
 \]  

Therefore, $f_{\delta} (z) \rightarrow f(z)$ uniformly on compact subsets of $\mathbb C_\i \setminus \partial U$ and $\|f_{\delta}\|_{L^\infty(\mu)} \le C_{32}.$There exists a subsequence $\{f_{\delta_n}\}$ such that $f_{\delta_n} (z) \rightarrow \tilde f(z)$ in $L^\infty(\mu)$ weak-star topology. Hence, $\|\tilde f\|_{L^\infty(\mu)} \le C_{32},$ $\tilde f\in R^{t,\i}(K,\mu),$ and $\mathcal C(f_{\delta_n} g\mu)(z) \rightarrow \mathcal C(\tilde fg\mu)(z), ~\area-a.a..$ From \eqref{BBFRLambdaEq1} and Corollary \ref{acZero}, we infer that
\[
\ f_{\delta_n}(z)\mathcal C(g\mu)(z) = \mathcal C(f_{\delta_n}g\mu)(z), ~\area-a.a..
\]
Thus, $f(z)\mathcal C(g\mu)(z) = \mathcal C(\tilde fg\mu)(z), ~\area-a.a.$ for $g\perp R^t(K,\mu).$ 
\end{proof}

\begin{proof} (Theorem A)
There exists a partition $\{\Delta_{00}, \Delta_{01}\}$ of $\text{spt}\mu$ such that
\[
 \ R^t(K, \mu ) = L^t(\mu |_{\Delta_{00}})\oplus R^t(K, \mu |_{\Delta_{01}})
 \]
 and $S_{\mu |_{\Delta_{01}}}$ is pure. Let $\mathcal F_{\Delta_{01}}$ be the non-removable boundary for $R^t(K, \mu |_{\Delta_{01}}).$ If $g\perp R^t(K, \mu ),$ then $g(z) = 0,~ \mu |_{\Delta_{00}}-a.a..$ Hence, by Theorem \ref{FCharacterization}, we see that $\mathcal F_{\Delta_{01}} \approx \mathcal F,~\gamma-a.a..$
 
  Let $\D(\lambda_0, \delta) \subset \text{abpe}(R^t(K, \mu ))$ such that for $\lambda\in \D(\lambda_0, \delta)$ and $r\in \text{Rat}(K),$ $|r(\lambda)| \le M \|r\|_{L^t(\mu)}$ for some $M > 0$ and $r(\lambda) = (r,k_\lambda),$  where $k_\lambda \in L^s(\mu)$ and $\|k_\lambda\|_{L^s(\mu)}\le M.$ Because $(z-\lambda)\bar k_\lambda\perp R^t(K, \mu ),$ we get $k_\lambda(z) = 0,~ \mu |_{\Delta_{00}}-a.a.$ if $\mu(\{\lambda\}) = 0.$ Let $\{\lambda_n\}$ be the set of atoms for $\mu.$ Then $|r(\lambda)| \le M \|r\|_{L^t(\mu|_{\Delta_{01}})}$ for $\lambda\in \D(\lambda_0, \delta) \setminus \{\lambda_n\}.$ Now for $|\lambda - \lambda_0| < \frac{\delta}{2},$ we have
  \[
  \ |r(\lambda)| \le \dfrac{C_{35}}{\pi\delta^2} \int_{\D(\lambda_0, \delta) \setminus \{\lambda_n\}} |r(z)| dm \le  C_{35} M \|r\|_{L^t(\mu|_{\Delta_{01}})}.
  \]
  Thus, $\lambda_0 \in  \text{abpe}(R^t(K, \mu|_{\Delta_{01}}))$ which implies $\text{abpe}(R^t(K, \mu|_{\Delta_{01}})) = \text{abpe}(R^t(K, \mu).$
  
  Therefore, the assumption of Theorem A holds for $R^t(K, \mu |_{\Delta_{01}})$ and we may assume that $S_\mu$ is pure in the following proof.

From \eqref{ABPETheoremRTEq1}, we see that $\partial U_i \cap \text{int}(K) \subset \mathcal F,~\gamma-a.a..$ Hence, the assumption $\partial U_i \cap \partial  K \subset \mathcal F,~\gamma-a.a.$ implies $\partial U_i \subset \mathcal F,~\gamma-a.a..$ 

Let $f_i = \chi_{U_i}.$ By Lemma \ref{ImRhoOnto2}, there exists $\tilde f_i \in R^{t,\i}(K,\mu)$ such that $\tilde f_i(z) = 0, ~ z\in \overline U_i^c$ and $f_i(z) \CT(g\mu)(z) = \CT(\tilde f_ig\mu)(z),~\area-a.a.$ for $g\perp \rtkmu.$ Thus, for $g\perp \rtkmu,$
\[
\ \CT(\tilde f_ig\mu)(z) = f_i(z)(f_i(z) \CT(g\mu)(z)) = f_i(z) \CT(\tilde f_ig\mu)(z) = \CT(\tilde f_i^2g\mu)(z),~\area-a.a.
\]
since $\tilde f_ig\perp \rtkmu.$ So, since $S_\mu$ is pure, we have
\[
\ \tilde f_i(z) = \tilde f_i^2(z),~\mu-a.a..
\]
Hence, there exists a Borel subset $\Delta_i$ with $U_i \subset \Delta_i \subset \overline U_i$ such that $\tilde f_i = \chi_{\Delta_i}.$ So
\begin{eqnarray}\label{TheoremA12Eq1}
\ \CT(\chi_{\Delta_i}g\mu)(z) = \chi_{U_i}(z)\CT(g\mu)(z),~\area-a.a., \text{ for } i \ge 1.
\end{eqnarray}
Therefore,
\[
\ \CT(\chi_{\Delta_i}\chi_{\Delta_m}g\mu)(z) = \chi_{U_i}(z)\chi_{U_m}(z)\CT(g\mu)(z) = 0,~\area-a.a.,
\]
for $i,m \ge 1$ and $i\ne m,$ which implies $\Delta_i \cap \Delta_m = \emptyset,~\mu-a.a.$ as $S_\mu$ is pure. Set $\Delta_0 = K\setminus \cup_{i=1}^\i \Delta_i.$ Then
\[
 \ R^t(K, \mu ) = R^t(K, \mu |_{\Delta_0})\oplus \bigoplus _{i = 1}^\infty R^t(K, \mu |_{\Delta_i}).
 \]
 
 For $i \ge 1,$ it is clear that $R^t(K, \mu |_{\Delta_i}) = R^t(\overline{U_i}, \mu |_{\Delta_i})$ and so (c) holds.
 
 Let $K_0$ be the spectrum of $S_{\mu |_{\Delta_0}}.$ 
 Then $R^t(K_0, \mu |_{\Delta_0}) = R^t(K, \mu |_{\Delta_0})$ and for $i \ge 1,$ $K_0 \cap U_i = \emptyset$ since, by \eqref{TheoremA12Eq1},
 \[
 \ \CT(\chi_{\Delta_0}g\mu)(z) = \CT(g\mu)(z) - \sum_{m=1}^\i \CT(\chi_{\Delta_m}g\mu)(z) = 0,~ \area|_{U_i}-a.a..
 \]
 Therefore, $\text{abpe}(R^t(K_0, \mu |_{\Delta_0})) = \emptyset.$ This proves (a).
 
For $f\in R^{t,\i}(K, \mu |_{\Delta_i})$ and $\lambda \in U_i,$ we have $\frac{f(z) - \rho_i(f)(\lambda)}{z - \lambda}\in R^{t,\i}(K, \mu |_{\Delta_i}).$ Hence, by Corollary \ref{acZero}, $\CT(fg\mu)(\lambda) = \rho_i(f)(\lambda)\CT(g\mu)((\lambda)),~\area-a.a.$ for $g\perp R^t(K, \mu |_{\Delta_i}).$ As $S_{\mu_{\Delta_i}}$ is pure, we conclude that $\rho_i$ is injective and 
\[
\ \rho_i(f_1f_2) = \rho_i(f_1)\rho_i(f_2),~f_1,f_2\in R^{t,\i}(K, \mu |_{\Delta_i}). 
\]

(b): If $B \subset \Delta_i$ and $\chi_B \in R^{t,\i}(K, \mu |_{\Delta_i})$ is a non-zero characteristic function. Then $(\rho_i(\chi_B))^2 = \rho_i(\chi_B),$ which implies $\rho_i(\chi_B) = \rho_i(\chi_{\Delta_i}) = 1$ since $U_i$ is connected. Hence, $\chi_B = \chi_{\Delta_i}$ as $\rho_i$ is injective. So $S_{\mu |_{\Delta_i}}$ is irreducible.

(d): For $f\in  R^{t,\i}(K, \mu |_{\Delta_i})$ and $\lambda\in U_i,$ there is a constant $C_\lambda > 0$ such that
\[
\ |\rho_i(f)(\lambda)| \le C_\lambda^{\frac 1n} \|f^n\|_{L^t(\mu |_{\Delta_i})}^{\frac 1n}. 
\]
Taking $n\rightarrow \infty,$ we get $\|\rho_i(f)\|_{U_i}\le \|f\|_{L^\infty(\mu |_{\Delta_i})}.$ 
From Lemma \ref{ImRhoOnto2}, we see $\rho_i$ is surjective. Moreover, $\|f^n\|_{L^\infty(\mu |_{\Delta_i})} ^{\frac 1n}\le C_{32} ^{\frac 1n} \|\rho_i(f^n)\|_{U_i}^{\frac 1n}.$ 
Letting 
$n\rightarrow \infty,$ we get $\|f\|_{L^\infty(\mu |_{\Delta_i})} \le \|\rho_i(f)\|_{U_i}.$ Thus,
\[
\ \|\rho_i(f)\|_{U_i} = \|f\|_{L^\infty(\mu |_{\Delta_i})},~ f\in \in R^{t,\i}(K_i,\mu |_{\Delta_i}).
\]
Therefore, $\rho_i$ is a bijective isomorphism between two Banach algebras $R^{t,\i}(K_i, \mu |_{\Delta_i})$ and $H^\infty (U_i)$. Clearly $\rho_i$ is also weak-star sequentially continuous, so an application of the Krein-Smulian Theorem shows that $\rho_i$ is a weak-star homeomorphism.

 \end{proof}
 
 \begin{proof} (Theorem B)
 Let $\text{abpe}(R^t(K, \mu)) = \cup_{i=1}^\i U_i,$ where $U_i$ is a connected component.
 By the assumption and Proposition \ref{RTFRProp2}, 
 \[
 \ \partial U_i \cap \partial  K \subset \partial_1  K \subset \mathcal F,~\gamma-a.a.. 
 \]
 Therefore, by Theorem A, we see that the decomposition in Theorem A holds. We only need to show that 
 \[
 \ R^t(K, \mu |_{\Delta_0}) = L^t(\mu |_{\Delta_0}).
 \]	
 In fact, from \eqref{ABPETheoremRTEq1}, if $\mathcal F_{\Delta_0}$ is the non-removable boundary for $R^t(K, \mu |_{\Delta_0}),$ then 
 \[
 \ \text{int}(K) \setminus \mathcal F_{\Delta_0} \approx \text{abpe}(R^t(K, \mu |_{\Delta_0})) = \emptyset,~ \gamma-a.a..
 \]
 Hence, $\mathcal F_{\Delta_0} = \C,~\gamma-a.a.$ since $\partial  K = \partial_1  K \subset \mathcal F_{\Delta_0},~\gamma-a.a.$ by Proposition \ref{RTFRProp2}. The proof follows from Corollary \ref{acZero}.
 \end{proof}

\bigskip

{\bf Acknowledgments.} 
The authors would like to thank Professor John M\raise.45ex\hbox{c}Carthy for carefully reading through the manuscript and providing many useful comments.
\bigskip

\bibliographystyle{amsplain}

\end{document}